\documentclass[12pt]{amsart}

\usepackage{graphicx,psfrag}
\usepackage{color,stmaryrd}

\usepackage{phaistos}
\usepackage{fancyhdr,color,graphicx}
\usepackage{tikz}
\usepackage{pgfplots}
\usepackage{psfrag}
\usepackage{multicol}
\usetikzlibrary{arrows,shapes,positioning,patterns}
\graphicspath{{graphics/}}
\pgfplotsset{compat=1.18}

\tikzset{
    cross/.pic = {
    \draw[rotate = 30] (-#1,0) -- (#1,0);
    \draw[rotate = 60] (0,-#1) -- (0, #1);
}
}

\newtheorem{thm}{Theorem}
\newtheorem{lem}[thm]{Lemma}

\newtheorem{question}{Question}
\newtheorem*{thm*}{Theorem}

\newcommand{\R}{\mathbb R}
\newcommand{\T}{\mathbb T}
\newcommand{\Z}{\mathbb Z}
\newcommand{\N}{\mathbb N}
\newcommand{\cA}{\mathcal A}
\newcommand{\Hom}{\mathcal H}
\newcommand{\interior}{\text{Int}}
\DeclareMathOperator{\diam}{diam}
\newcommand{\cP}{\mathcal P}

\title{Realization of Anosov Diffeomorphisms on the Torus}
\author{Tamara Kucherenko}\address{Department of Mathematics,
The City College of New York, New York, NY, 10031, USA}\email{tkucherenko@ccny.cuny.edu}

\author{Anthony Quas}\address{ Department of Mathematics and Statistics, University of Victoria, Victoria, BC
Canada}\email{aquas@uvic.ca}

\thanks{T.K. is supported by grants from the Simons Foundation \#430032. }
\thanks{A.Q. is supported by a grant from NSERC}

\begin{document}
\begin{abstract}
We study area preserving Anosov maps on the two-dimensional torus within a fixed 
homotopy class.  We show that the set of pressure functions for Anosov diffeomorphisms 
with respect to the geometric potential is equal to the set of pressure functions 
for the linear Anosov automorphism with respect to H\"{o}lder potentials. We use this 
result to provide a negative answer to the $C^{1+\alpha}$ version of the question 
posed by Rodriguez Hertz on whether two homotopic area preserving $C^\infty$ 
Anosov difeomorphisms whose geometric potentials have identical pressure 
functions must be $C^\infty$ conjugate.


\end{abstract}

\maketitle
\section{Introduction}
We consider an Anosov diffeomorphism $T$ of the two-dimensional torus $\T^2$. 
That is, there is a continuous splitting of the tangent bundle of $\T^2$ into 
a direct sum $E^u\oplus E^s$ which is preserved by the derivative $DT$ and such that the unstable 
subbundle $E^u$ is uniformly expanded by $DT$ and the stable subbundle $E^s$ is uniformly 
contracted by $DT$. Any such Anosov diffeomorphism $T$ is homotopic and 
topologically conjugate to a hyperbolic toral automorphism $L$ given by an 
integer matrix with determinant one and no eigenvalues of absolute value one. 
This was first proven by Franks in 1969 \cite{Franks} under the assumption that all 
points on the torus are non-wandering (in fact, his result was for an $n$-dimensional torus). 
A year later Newhouse \cite{Newhouse} pointed out that this assumption is satisfied 
when either $\dim E^s=1$ or $\dim E^u=1$, which provided the classification of Anosov 
diffeomorphisms up to topological conjugacy in dimensions 2 and 3. 
The case of dimension $n\ge 4$ was settled by Manning \cite{Manning} in 1974.

Suppose $T_1$ and $T_2$ are two $C^r \,(r>1)$ Anosov diffeomorphisms in the homotopy 
class of a fixed hyperbolic automorphism $L$. It follows from the above that there 
is a homeomorphism $h$ such that $h\circ T_1=T_2\circ h$. The problem of determining 
when $h$ has the same regularity as the maps $T_1$ and $T_2$ is known as the smooth 
conjugacy problem and has been studied extensively, see e.g. \cite{LMM-0,KaSa,Go, GoRH}. 
Already in 1967 Anosov \cite{Anosov} constructed examples which showed that $h$ 
may be merely H\"older even for highly regular $T_1$ and $T_2$, which initially discouraged 
further study of the problem (see comments in \cite{Sm}). However, a series of papers 
\cite{LMM-I,LMM-II,LMM-III,LMM-IV}, authored, in various combinations, by de la 
Llave, Marco, and Moriy\'on, appeared in the 1980s focusing on the study of the 
conjugacy of $C^\infty$ diffeomorphisms on $\T^2$.
The culmination of their work is the following theorem.
\begin{thm*}{\cite{LMM-IV}} Let $T_1$ and $T_2$ be $C^\infty$ Anosov diffeomorphisms 
of $\T^2$. If they are topologically conjugate and the Lyapunov exponents at 
corresponding periodic orbits are the same, then the conjugating homeomorphism is $C^\infty$.
\end{thm*}
Later it was shown that the equality of the corresponding Lyapunov exponents for 
$C^r$ Anosov diffeomorphisms on $\T^2$ implies that the conjugacy is $C^{r-\epsilon}$, 
however it is no longer true on $\T^4$ even for $C^\infty$ maps \cite{LMM-V}. 
The case of $\T^3$ is still open, with a positive result recently obtained 
when one of the diffeomorphisms is an automorphism \cite{DeWGigolev}.

Note that if $h$ is differentiable, then for any point $x$ of period $n$ for 
$T_1$, $h(x)$ is of period $n$ for $T_2$ and 
$$
DT_1^n(x)=Dh^{-1}(h(x))DT_2^n(h(x))Dh(x).
$$ 
We see that the Lyapunov exponents of $x$ under $T_1$ and $h(x)$ under $T_2$ coincide. 
The result of \cite{LMM-IV} is quite remarkable since a condition, which is a priori 
weaker than $h$ being $C^1$, is shown to imply that $h$ is $C^\infty$.
F. Rodriguez Hertz asked whether we can get away with even less. He proposed to 
replace the assumption of equality of the Lyapunov exponents by the equality of the 
pressure functions of the geometric potentials.

To introduce the pressure function we first define the topological pressure using the 
variational principle. The topological pressure of a continuous potential $\phi:\T^2\to\R$ 
with respect to a dynamical system $T:\T^2\to\T^2$ is given by
$$P_{\rm top}(T,\phi)=\sup_\mu\left\{h_\mu(T)+\int\phi\,d\mu  \right\},$$
where $\mu$ runs over the set of all $T$-invariant probability measures on $\T^2$ and 
$h_{\mu}(T)$ is the measure-theoretic entropy of $\mu$. A measure $\mu$ which 
realizes the supremum is called an equilibrium state of $\phi$. By a celebrated 
result of Bowen \cite{Bowen}, for an Anosov diffeomorphism $T$ any H\"{o}lder 
potential $\phi:\T^2\to\R$ has a unique equlibrium state $\mu_\phi$. Equilibrium 
states are mathematical generalizations of Gibbs distributions in statistical physics. 
The most important ones are the measure of maximal entropy, which is the equilibrium 
state of a constant potential, and the SRB measure, which is the equilibrium state of the 
\emph{geometric potential}. The geometric potential is the negative logarithm of the 
Jacobian of $T$ along the unstable bundle $E^u$,
$$
\phi_{T}^{u}(x)=-\log\big|D_u T(x)\big|.
$$

The \emph{pressure function} of a potential $\phi$ is the map $t\mapsto P_{\rm top}(T,t\phi)$, 
where $t$ is a real valued parameter. Information about various dynamical properties 
of an Anosov system is encoded into the pressure function of the geometric potential. 
For example, when $T$ is area preserving, the positive Lyapunov exponent of $T$ with 
respect to the normalized Lebesgue measure (which is the equilibrium state of $\phi_{T}^{u}$) 
is given by the negative derivative of the pressure function of $\phi_{T}^{u}$ at $t=1$, 
while the derivative at $t=0$ gives the Lyapunov exponent with respect to the measure of 
maximal entropy of $T$.
F. Rodriguez Hertz asked whether information on the regularity of the conjugating 
homeomorphism can also be extracted from the pressure functions of the geometric 
potentials of the corresponding maps. More precisely,
\begin{question}{\cite[attr. F. Rodriguez Hertz]{E2}}\label{q:Federico}
Let $T_1$ and $T_2$ be $C^\infty$ area-preserving Anosov diffeomorphisms on $\T^2$ that are homotopic. Assume $P_{\rm top}(T_1,t\phi^u_{T_1})=P_{\rm top}(T_2,t\phi^u_{T_2})$ for all $t$. Does this imply that $T_1$ and $T_2$ are $C^\infty$ conjugate?
\end{question}

We point out that the answer to the above question is positive when one of the 
diffeomorphisms is an automorphism. Indeed if $T_1$ is an automorphism, then
$\phi^u_{T_1}$ is constant, so that $P_{\rm top}(T_1,t\phi^u_{T_1})$, and hence 
$P_{\rm top}(T_2,t\phi^u_{T_2})$ is affine. However, pressure functions of H\"older
continuous functions are known to be strictly convex unless the underlying potential is
cohomologous to a constant. Hence $\phi^u_{T_2}$ is cohomologous to the constant $\phi^u_{T_1}$.
This guarantees that the Lyapunov exponents of periodic points of $T_2$ match those of
periodic orbits of $T_1$, so that $T_1$ and $T_2$ are $C^\infty$ conjugate by the above result.


One reason that Anosov diffeomorphisms on $\T^2$ are well-understood is that they admit 
symbolic codings. Using a Markov partition of $\T^2$ one can find a finite set $\cA$ 
(indexing the set of rectangles of the Markov partition) and a mixing subshift of finite type 
$\Omega\subset \cA^\Z$ such that there exists a finite-to-one factor map 
$\pi:\Omega\to \T^2$ which is H\"older. Then $\phi^u_T\circ\pi$ is a H\"older 
potential on $\Omega$. 
It turns out that in the symbolic setting, a related question to Question \ref{q:Federico} 
has been studied by Pollicott and Weiss in \cite{PoW}. 

Suppose $(\Omega,\sigma)$ is a subshift of 
finite type and $\psi:\Omega\to\R$ is a H\"older potential. Denote the Birkhoff sum of 
$\psi$ by $S_n\psi(x)=\sum_{k=0}^{n-1}\psi(\sigma^k x)$. The multi-set 
$\{(S_n\psi(x),n): \sigma^n x=x\}$ is called the \emph{unmarked orbit spectrum of $\psi$}. 
In \cite{PoW} the extent to which a potential is determined by its periodic orbit 
invariants such as its orbit spectrum and its pressure function was investigated. 
Note that for subshifts of finite type the pressure function can be defined 
topologically as
$$P_{\rm top}(\sigma,t\psi)=\lim_{n\to\infty}\frac{1}{n}\log\left(\sum_{\sigma^nx=x}e^{tS_n\psi(x)}\right),$$
and therefore any two potentials with the same unmarked orbit spectrum must have 
identical pressure functions. The converse is not true. It was shown by Pollicott 
and Weiss that there exists an uncountable family of H\"older
continuous functions on a full shift with different unmarked orbit spectra, 
but all sharing the same pressure function.

Since for Anosov $T:\T^2\to\T^2$ we have $-\log\big|D_uT^n\big|=S_n\phi^u_T$, 
the equality of the Lyapunov exponents at periodic orbits for torus diffeomorphisms 
$T_1$ and $T_2$ corresponds to the equality of the unmarked orbit spectra of their 
geometric potentials. Hence Question \ref{q:Federico} may be seen as
asking whether H\"older functions arising from geometric potentials of Anosov 
diffeomorphisms on the torus are special enough that the equality of their pressure 
functions implies the equality of their unmarked orbit spectrum. That turns out not to be the case.

 We show that the set of pressure functions for Anosov diffeomorphisms with respect 
 to their geometric potentials is equal to the set of pressure functions for the hyperbolic 
 automorphism with respect to H\"{o}lder potentials.  

\begin{thm}\label{thm:main}
    Let $L$ be a hyperbolic automorphism of $\T^2$ and let $\mu$ be the equilibrium state for a
    H\"older continuous potential $\phi$ with $P_{\rm top}(L,\phi)=0$.
    Then there exists a $C^{1+H}$ area-preserving Anosov diffeomorphism $T$ of $\T^2$ such that  
    \begin{itemize}
\item   the system $T\colon (\T^2,\mathsf{Leb})\to (\T^2,\mathsf{Leb})$ is
    conjugate to $L\colon (\T^2,\mu)\to(\T^2,\mu)$ by a map $h$;
    \item     the potential  $-\log|D_uT|\circ h$ is cohomologous to $\phi$.
    \end{itemize}
\end{thm} 
In this theorem, and throughout the paper, we write $T$ is $C^{1+H}$ to mean 
that there exists $0<\alpha<1$ where $T$ is $C^{1+\alpha}$. 

A statement similar to the above theorem could be deduced from the work by Cawley 
\cite{C} which establishes a bijection between Teichm\"{u}ller space of an Anosov 
diffeomorphism and the quotient of H\"older functions by the subspace of coboundaries 
plus constants. However the proofs in \cite{C} appear to be rather opaque. Our approach 
is constructive where the main step -- the change of coordinates -- is given by an 
explicit formula in terms of the equilibrium state of $\phi$. 


In view of Theorem \ref{thm:main}, to solve Question \ref{q:Federico} we need to find 
H\"{o}lder potentials having identical pressure functions with respect to an 
automorphism $L$, but different unmarked orbit spectra. From the work of Pollicott 
and Weiss one might expect uncountably many such potentials on the corresponding 
subshift of finite type. However, there is no reason to expect that any of these 
potentials will be H\"older continuous on the torus. Hence we have to employ 
another construction to produce torus continuous examples. We obtain
\begin{thm}\label{thm:Federico}
There exist homotopic $C^{1+H}$ area-preserving Anosov diffeomorphisms 
$T_1$ and $T_2$ on $\T^2$ such that $P_{\rm top}(T_1,t\phi^u_{T_1})=P_{\rm top}(T_2,t\phi^u_{T_2})$
for all $t$, but  $T_1$ and $T_2$ fail to be $C^1$ conjugate.
\end{thm}

In fact our results give countably many homotopic H\"older differentiable area-preserving Anosov diffeomorphisms, none of which are $C^1$ conjugate, 
but all having the same 
pressure function. We do not know whether one can find uncountably many such maps, 
as would be suggested by the result in \cite{PoW}.

We remark that our examples, which are in the $C^{1+H}$ category, do not
directly respond to the $C^\infty$ question of Rodriguez Hertz; however they strongly suggest
a negative answer to that question also. \\

 \textbf{Acknowledgement.} Part of this work was completed during our one-week stay at the 
 \textsl{Centre International de Rencontres Mathématiques} in Luminy, France through the 
 \textsl{Research in Residence} program. We thank CIRM for the support and hospitality.

\section{Preliminary Results}
\subsection{Gibbs Measures and Radon-Nikodym Derivative}\label{sec:R-N_deriv}

In recent works an invariant measure is termed Gibbs if the weight of the Bowen balls 
of order $n$ satisfies the growth estimate given  in \cite[Theorem 1.2]{Bowen}. 
We recall the original definition of a Gibbs state introduced by 
Ruelle \cite{Ruelle} and Capocaccia \cite{Ca}, which is equivalent to  
Bowen's property from \cite{Bowen} in our situation. Let $T:M\to M $ be an 
expansive homeomorphism on a compact metric space $M$. A map $\chi$ 
from some open set $U\subset M$ into $M$ is called \emph{conjugating} for the 
system $(M,T)$ if $d(T^n\circ\chi(x),T^n(x))\to 0$ for $|n|\to\infty$ 
uniformly in $x\in U$. 
In the case of an 
Anosov automorphism $L$, the conjugating homeomorphisms are locally given by
$x\mapsto x+v$ where $v$ is homoclinic to 0. For this article, we only need the global 
conjugating homeomorphisms $x\mapsto x+v$.

Suppose $\phi$ is a continuous function on $M$. A probability measure $\mu$ on $M$ is a \emph{Gibbs state}
for $\phi$ if for every conjugating homeomorphism $\chi:U\to \chi(U)$ where $U=U_\chi$ is
an open set in $M$ the measure $\chi_*( \mu|_{U} )$ is absolutely continuous 
with respect to $\mu|_{\chi(U)}$,
with Radon-Nikodym derivative
\begin{equation}\label{eq:RNderivative}
  \frac{d\chi_*\mu}{d\mu}=\exp \sum_{n\in\Z}\big[\phi\circ T^n\circ \chi^{-1}-\phi\circ T^n\big].
\end{equation}
For an axiom A diffeomorphism the equilibrium state of a H\"{o}lder potential $\phi$ is also a
Gibbs state for $\phi$, which is proven in Ruelle's book \cite[Theorem 7.18]{Ruelle}. 
A result of Haydn \cite{Haydn} is that the converse holds as well. In fact, Haydn and Ruelle 
show in \cite{HR}
that equilibrium states and Gibbs states are equivalent for expansive homeomorphisms with specification
and Bowen potentials.

We need the regularity properties of the Radon-Nikodym derivative (\ref{eq:RNderivative}).
Although the question of regularity seems to be very natural, we were not able to locate a
corresponding result in the literature. We provide a proof in the case of Anosov
automorphisms, however the same
argument can be straightforwardly generalized to
Anosov diffeomorphisms, Axiom A diffeomorphisms or more general Smale spaces. 

\begin{lem}\label{lem:rnderivHolder}
    Let $L:\T^2\to\T^2$ be an Anosov automorphism, let $v$ be homoclinic to 0 and let 
    $\tau(x)=x-v$. Let $\phi$ be a H\"older continuous function and let $\mu$ be the 
    corresponding equilibrium state. 
    Then the Radon-Nikodym derivative $\frac{d\tau_*\mu}{d\mu}$ in \eqref{eq:RNderivative}
    above is H\"older continuous.
\end{lem}

\begin{proof}
Let $\lambda$ be the expanding eigenvalue of $L$. Then there exist $C_1$ and $C_2$ such that
$d(L^nv,0)\le C_1\lambda^{-|n|}$ and $d(L^nx,0)\le C_2\lambda^{|n|}d(x,0)$
for all $n\in\Z$.
We let $C_3>0$ and $\alpha\in(0,1)$ be such that $|\phi(x)-\phi(y)|\le C_3d(x,y)^\alpha$ for all $x,y\in\T^2$. 

We define
$$
\theta(x)=\sum_{n\in\Z}\left[\phi(L^n(x+v))-\phi(L^nx)\right].
$$
Suppose $x,y\in\T^2$ satisfy $d(x,y)<\lambda^{-2k}$ for some $k$. Then we calculate
\begin{align*}
    |\theta(y)&-\theta(x)|
    \le\sum_{n\in\Z}
    \big|\phi(L^n(y+v))-\phi(L^ny)-\phi(L^n(x+v))+\phi(L^nx)\big|\\
    &\le \sum_{|n|\le k}\big[|\phi(L^n(y+v))-\phi(L^n(x+v))|+
    |\phi(L^n(y))-\phi(L^n(x))|\big]\\
    &+\sum_{|n|>k}\big[|\phi(L^n(y+v))-\phi(L^n(y))|+|\phi(L^n(x+v))-\phi(L^n(x))|\big].
\end{align*}
We bound the sums by geometric series and obtain
\begin{equation*}
    |\phi(L^ny)-\phi(L^nx)|\le C_3(C_2\lambda^{|n|}d(x,y))^\alpha\le C_3C_2^\alpha
    \lambda^{-(2k-|n|)\alpha},
\end{equation*}
with the same bound for $|\phi(L^n(y+v))-\phi(L^n(x+v))|$.
Likewise, $|\phi(L^n(x+v))-\phi(L^n(x))|\le C_3C_1^\alpha\lambda^{-|n|\alpha}$
with the same bound for $|\phi(L^n(y+v))-\phi(L^n(y))|$. Summing the geometric series, 
we obtain 
$|\theta(y)-\theta(x)|\le K\lambda^{-k\alpha}$,
where $K=4C_3(C_1^\alpha+C_2^\alpha)/(1-\lambda^{-\alpha})$
showing that $\theta$ is H\"older as required.     
\end{proof}

\subsection{Coding for Toral Automorphisms}\label{sec:coding}

Let $L$ be a mixing toral automorphism of $\T^2$ and we let $\mathcal P$ be a generating Markov partition,
which we assume to consist of (closed) rectangles whose boundaries 
are pieces of the unstable and stable manifolds
through the origin. We make the further assumption that if $A$ and $B$ are elements of the partition, then
$(A+v)\cap B$ is connected (either a rectangle or an empty set). This condition is automatically satisfied if
$\diam(\mathcal P)<\frac 12$, and so may be assumed without loss of generality by replacing
$\mathcal P$ with a Markov partition of the form $\bigvee_{j=0}^{m-1} L^{-j}\mathcal P$ if necessary.

For $\cA=\{0,...,\# (\cP)-1\}$ let $\Omega\subset \cA^\Z$ be the corresponding shift of finite type and let $\pi\colon \Omega\to \T^2$ be the
corresponding finite-to-one factor map from $(\Omega,\sigma)$ to $(\T^2,L)$.  The map $\pi$ is one-to-one on a
set of measure 1 with respect to any invariant measure on $\Omega$.
We equip $\Omega$
with the standard metric on $\Omega$ where $d(\omega,\omega')=2^{-n}$ 
if $\omega_j=\omega'_j$ whenever $|j|<n$, but
$\omega_{\pm n}\ne \omega'_{\pm n}$.

If $\phi$ is a H\"older continuous function on $\T^2$, we let $\mu$ be its equilibrium measure. We also
set $\psi=\phi\circ\pi$ to be the corresponding potential on $\Omega$ and let $\nu$ be the equilibrium
measure of $\psi$. Since $\pi$ is one-to-one $\nu$-almost everywhere, 
$\pi_*\nu=\mu$. Let $\Omega^+\subset \cA^{\N_0}$ be the one-sided version of $\Omega$
that is the image of $\Omega$ under the map $p_+\colon \cA^\Z\to \cA^{\N_0}$ defined by 
$p_+(\omega)_n=\omega_n$ for $n\ge 0$.
Similarly, let $\Omega^-\subset \cA^{-\N}$ be the image of $\Omega$ under the restriction 
map $p_-\colon \cA^\Z\to \cA^{-\N}$. Then $\nu^+=(p_+)_\ast\nu$ and $\nu^-=(p_-)_\ast\nu$ are the measures corresponding to $\nu$ on $\Omega^+$ and $\omega^-$ respectively.

The main symbolic result we are using 
is the local product structure of $\nu$. Ruelle proves in \cite[Lemma 5.9]{Ruelle} 
that $\nu$ has \emph{local product structure}, i.e. 
$d\nu(\omega)=\hat{\varrho}(\omega)\,d\hat{\nu}^+(p_+(\omega))\,d\hat{\nu}^-(p_-(\omega))$
where $\hat{\nu}^+$ is a probability measure on $\Omega^+$, $\hat{\nu}^-$ is a probability 
measure on $\Omega^-$, and $\hat{\varrho}$ is a positive continuous function on $\Omega$. 
Furthermore, it is shown in \cite[Lemma 5.23]{Ruelle} that $\hat{\varrho}$ is H\"older on 
$\Omega$, and the functions  $\hat{\varrho}^+(\omega^+)=
\int\hat{\varrho}(\omega)\, d\hat{\nu}^-(\omega^-)$, 
$1/\hat{\varrho}^+(\omega^+)$ are H\"older on $\Omega^+$. 
Analogous statements hold for $\hat{\varrho}^-(\omega^-)$. 
Note that for each $\omega^+\in\Omega^+$ the integral is taken 
over the set $\{\omega^-\in \Omega^-\colon \omega^-_{-1}\omega^+_0
\text{ is legal in $\Omega$}\}$. 
In this case the measure $\nu^+$ on $\Omega^+$ is 
given by $d\nu^+=\varrho^+(\omega^+)\,d\hat{\nu}^+$; similarly for $\nu^-$. 

We are mostly concerned with the structure of $\nu$ on the cylinder 
$[0]=\{\omega\in\Omega: \omega_0=0\}$. We let 
$A^-=\{\omega^-\in \Omega^-\colon \omega^-_{-1}0\text{ is legal in $\Omega$}\}$. 
For $\omega^-\in A^-$ and $\omega^+\in p_+([0])$ we write 
$\varrho^+(\omega^+)=\int_{A^-}\hat{\varrho}(\omega^-\omega^+)\, d\hat{\nu}^-(\omega^-)$, 
$\varrho^-(\omega^-)=\int_{[0]}\hat{\varrho}(\omega^-\omega^+)\, d\hat{\nu}^+(\omega^+)$, 
and   $\varrho(\omega^-\omega^+)=\frac{\hat{\varrho}(\omega^-\omega^+)} 
{\varrho^-(\omega^-)\varrho^+(\omega^+)}$, 
so that $d\nu(\omega)=\rho(\omega)\,d\nu^+(\omega^+)\,d\nu^-(\omega^-)$. 
In particular,  
\begin{equation}\label{eq:marginal}
\begin{split}
    \int_{A^-}\varrho(\omega^-\omega^+)\,d\nu^-(\omega^-)&= 
    \int_{A^-}\frac{\hat{\varrho}(\omega^-\omega^+)} 
    {\varrho^-(\omega^-)\varrho^+(\omega^+)}\,d\nu^-(\omega^-)\\
    &=\frac{1}{\varrho^+(\omega^+)}\int_{A^-}\hat{\varrho}(\omega^-\omega^+)\,d\hat{\nu}^-(\omega^-)\\
    &=1
\end{split}
\end{equation}
We summarize the above in the following lemma which is frequently used throughout this article.

\begin{lem}[Ruelle \cite{Ruelle}]\label{lem:Ruelle}
Let $\psi$ be a H\"older continuous function on a mixing shift of finite type $\Omega$ and let $\nu$ be
its equilibrium state.
Then $\nu$ has \emph{local product structure}. That is, on the cylinder set $[0]$ 
there exist a positive H\"older continuous function $\varrho(\omega)$
such that $d\nu(\omega)=\varrho(\omega)\,d\nu^+(\omega^+)\,d\nu^-(\omega^-)$ 
where $\nu^-$, $\nu^+$ are the restrictions of $\nu$ to $\Omega^+$, $\Omega^-$ respectively,
and $\omega$ denotes the concatenation of $\omega^-$ and $\omega^+$.
\end{lem}

It is shown by Walters in \cite{Walters} that under the assumptions of the above lemma there is a H\"older function $g:\Omega^+\to (0,1)$ such that $\log g$ is cohomologous to $\phi$ and  $\nu^+$ is the unique $g$-measure for $g$, i.e. for $\omega^+\in\Omega^+$
\begin{equation}\label{eq:def_g-func}
    g(\omega^+)=\lim_{\substack{{\rm diam}(S)\to 0\\\nu^+(S)\ne 0,\, \omega^+ \in S}}\frac{\nu^+(S)}{\nu^+(\sigma_+(S))}.
\end{equation}


Since the map $\pi\colon\Omega\to\T^2$ is H\"older continuous, given a H\"older 
continuous function $\phi$ on the torus,
we see that $\phi\circ\pi$ is H\"older; however many H\"older continuous functions on the shift cannot be
written in the form $\phi\circ\pi$. We call a function $f$ defined on $\Omega$ 
\emph{torus-H\"older} if it can be
written in the form $\phi\circ\pi$ where $\phi$ is a H\"older continuous function of the torus. 
A subset $R$ of $\Omega$ is called a \emph{rectangle} if it satisfies the following  conditions
\begin{itemize}
    \item $\omega,\omega'\in R$ implies
    the concatenation $p_-(\omega)p_+(\omega')$ belongs to $R$;
    \item $\pi(R)$ is connected;
    \item $\diam(\pi(R))<\frac 12$;
     \item $R=\pi^{-1}(\pi (R))$.
\end{itemize}


\begin{lem}
Let $L$ be an Anosov automorphism of $\T^2$ and let $\mathcal P$ be a Markov partition as described above.
Let $\Omega$ be the corresponding shift of finite type and let $\pi\colon\Omega\to\T^2$ be the natural
factor map.
Let $R$ be a rectangular subset of a cylinder set $[i]$ in $\Omega$
and suppose that $f\colon R\to\R$ is a H\"older continuous function.
If $f$ has the property that $f(\omega)=f(\omega')$ whenever
$\pi(\omega)=\pi(\omega')$, then $f$ may be expressed as
$h\circ\pi$ where $h$ is a H\"older continuous function defined on $\pi(R)\subset\T^2$.
\end{lem}

\begin{proof}
Since $f(\omega)=f(\omega')$ when $\pi(\omega)=\pi(\omega')$, we see that $f$ takes the same
value on each element of $\pi^{-1}(x)$ for any $x\in\pi(R)$. Hence $h(x):=f(\pi^{-1}x)$ is
well-defined on the rectangle $A:=\pi(R)$ which has sides parallel to the stable and unstable directions.
Since $f$ is H\"older continuous, let $c$ and $\alpha$ be such that
$|f(\omega)-f(\omega')|\le c\alpha^n$ whenever $d(\omega,\omega')\le 2^{-n}$.

Since $A$ is a rectangle in $\T^2$, 
we define for $x,y\in A$, 
$\llbracket x,y\rrbracket_A$ to be the unique point
$z$ in $A$ such that the line segments $[x,z]$ and $[z,y]$ lie in $A$ with $[x,z]$ in the stable
direction and the $[z,y]$ in the unstable direction. 
We now estimate $|h(x)-h(z)|$. An exactly similar estimate applies to $|h(z)-h(y)|$.
Let $C$ be the constant (depending only on the angle between the stable and unstable directions)
so that if $x,y$ lie in $A$ then $d(x,\llbracket x,y\rrbracket_A),d(y,\llbracket
x,y\rrbracket_A)\le Cd(x,y)$.
Let $\lambda$ be the expanding eigenvalue and let $n$ be the smallest natural number such that
$C^{-1}\diam(\mathcal P)\lambda^{-n}\le d(x,y)$.

Let $x=\pi(\xi)$ and $\llbracket x,y\rrbracket_A=\pi(\zeta)$.
Then either $x$ and $\llbracket x,y\rrbracket_A$ lie in the same element of
$L^{j}\mathcal P$ for each $0\le j<n$, in which case $|h(x)-h(\llbracket x,y\rrbracket_A)|
=|f(\xi)-f(\zeta)|\le c\alpha^n$
or there exists a point $w$ in $\partial L^{-(n-1)}\mathcal P\cap [x,\llbracket x,y\rrbracket_A]$.
Since $d(x,w)$ and $d(\llbracket x,y\rrbracket_A,w)$ are less than 
$\diam(\mathcal P)\lambda^{-(n-1)}$ and $w$ is on the boundary,
$x$ and $w$ must belong to a common element of
$L^{-(n-1)}\mathcal P$ and similarly for $w$ and $\llbracket x,y\rrbracket_A$, see Fig \ref{fig:xwz}.

\begin{figure}[ht]
\begin{tikzpicture}
  \begin{axis} [clip=false,
  x=0.5cm,y=0.5cm,
  axis y line=left,
  axis x line=middle,
  ytick=\empty,
  xtick=\empty,
  ymin=-0.5, ymax=8.0,
  xmin=-0.5,xmax=8.0,
  axis line style={draw=none}
  ] \node[anchor=east] (source) at (axis cs:0,8.0){\small{$\T^2$}};
    \node[anchor=east] (source) at (axis cs:4.0,6.5){\small{$A_0$}};
    \draw[very thick](axis cs:0,0) -- (axis cs:8,0) -- (axis cs:8,8) -- (axis cs:0,8) -- (axis cs:0,0);
    \draw[thick](axis cs:2,1.236) -- (axis cs:7.5,0.618*7.5) -- (axis cs:5.8,2.236*7.5-1.618*5.8) -- (axis cs:0.3,4.472-1.618*0.3) -- (axis cs:2,1.236);
    \filldraw (axis cs:2.2,3) circle (2pt);
    \node[anchor=south] (source) at (axis cs:2.2,3){\small{$x$}};
    \draw[thick] (axis cs:2.2,3) -- (axis cs:5.5,0.618*5.5+1.6404);
    \filldraw (axis cs:5.5,0.618*5.5+1.6404) circle (2pt);
    \node[anchor=south] (source) at (axis cs:5.5,0.618*5.5+1.6404){\small{$\llbracket x,y\rrbracket_A$}};
    \filldraw (axis cs:4,0.618*4+1.6404) circle (2pt);
    \node[anchor=south] (source) at (axis cs:4.1,0.618*4.1+1.6404){\small{$w$}};
    \draw[thick] (axis cs:3.05,1.618*.95+4.1124) -- (axis cs:4.7,-1.618*0.7+4.1124);
 \end{axis}
\end{tikzpicture}
\caption{On $\T^2$ the unstable 
and stable directions are shown as north-east and north-west
respectively.}
\label{fig:xwz}
\end{figure}
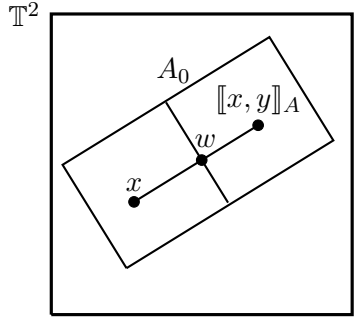


Now write $w=\pi(\eta)=\pi(\eta')$ where
$\eta_{-\infty}^{n-1}=\xi_{-\infty}^{n-1}$ and ${\eta'}_{-\infty}^{n-1}=\zeta_{-\infty}^{n-1}$.
We then have $|h(x)-h(z)|=|f(\xi)-f(\zeta)|\le |f(\xi)-f(\eta)|+|f(\eta')-f(\zeta)|\le 2c\alpha^n,$
where we made use of the fact that $f(\eta)=f(\eta')$. Combining this with the analogous estimate
for $|h(\llbracket x,y\rrbracket_A)-h(y)|$, we see $|h(x)-h(y)|\le 4c\alpha^n\le 4c
\big(C d(x,y)/(\lambda\diam(\mathcal P)\big)^{-\log\alpha/\log\lambda}$, so that $h$
is H\"older as required.
\end{proof}

\section{Anosov realization}

In this section we show that given a hyperbolic automorphism $L$ for any positive 
H\"{o}lder continuous potential $\phi$ with zero topological pressure there exists 
a conjugate Anosov diffeomorphism $T$ for which the geometric potential is cohomologous to $\phi$.

\begin{thm}
    Let $L$ be an Anosov automorphism of $\T^2$ and let $\mu$ be the equilibrium state for a
    H\"older continuous potential $\phi$ with $P_{\rm top}(L,\phi)=0$.
    Then there exists a $C^{1+H}$-atlas on $\T^2$ with respect to which $L$ is an 
    Anosov diffeomorphism with H\"older derivative and its geometric potential is cohomologous to $\phi$.    
\end{thm}

We prove the theorem in a number of steps.

\subsection{Definition of new $C^{1+H}$ atlas}\label{sec:atlas}
We let $\Hom$ denote the collection of points of $\T^2$ that are homoclinic to 0 under the action of $L$.
Since $L$ is an automorphism, it follows that if $v\in\Hom$ and $x\in \T^2$ then
$d(L^n(x+v),L^n(x))=d(L^nv,0)\to 0$ as $|n|\to\infty$. Recall that the points 
homoclinic to 0 are dense in $\T^2$ (see e.g. \cite{Mane}).

For the remainder of this section $A_0$ denotes the element of the partition $\cP$ 
which corresponds to the cylinder $[0]$ in $\Omega$, i.e. $\pi([0])=A_0$.

\begin{lem}\label{lem:bracket}
Let $w\in\Hom$ and suppose that $A_0\cap (A_0-w)$ has non-empty interior.
Then there exist vectors $u,v\in\Hom$ such that:
\begin{itemize}
\item $u+v=w$;
\item if $x\in\interior(A_0\cap (A_0-u))$ then the
line segment $[x,x+u]$ lies in $\interior(A_0)$ and is parallel to the stable direction;
\item if $x\in\interior(A_0\cap (A_0-v))$ then the
line segment $[x,x+v]$ lies in $\interior(A_0)$ and is parallel to the unstable direction;
\item $\interior(A_0\cap (A_0-w))=\interior\big(A_0\cap (A_0-u)\cap (A_0-v)\big)$.
\end{itemize}
\end{lem}

\begin{proof}
For any $x\in \interior(A_0\cap (A_0-w))$, since $A_0$ is a parallelogram
with edges parallel to the stable and unstable directions, the vector $w$ may
be expressed as a sum of pieces $u$ and $v$ parallel to the stable and unstable
directions, where $[x,x+u]$ and $[x,x+v]$ lie in $A_0$. Note that $x+u$ is the point of 
intersection of the stable manifold of $x$ and the unstable manifold of $x+w$. 
Linearity of $L$ implies that $u$ belongs to the stable manifold of 0 and unstable 
manifold of $w$. Since $w\in\Hom$, we conclude that $u\in\Hom$ as well. 
Similarly, $x+v$ is the point of intersection of the unstable 
manifold of $x$ and the stable manifold of $x+w$, so $v\in\Hom$.
\end{proof}


We define two functions $\xi_1$ and $\xi_2$ on $A_0$. Let $\xi_1(x)$ be the $\mu$-measure
of the rectangle contained in $A_0$
lying to the left of the connected portion of the stable manifold of $x$ within $A_0$ as illustrated in
Figure \ref{fig:xi1}. Similarly, let $\xi_2(x)$ be the $\mu$-measure of the rectangle 
contained in $A_0$ lying below the
connected portion of the unstable manifold of $x$ within $A_0$. We denote $\xi(x)=(\xi_1(x),\xi_2(x))$.

\begin{figure}[ht]
\begin{tikzpicture}[fill opacity=.7]
  \begin{axis} [clip=false,
  x=0.5cm,y=0.5cm,
  axis y line=left,
  axis x line=middle,
  ytick=\empty,
  xtick=\empty,
  ymin=-0.0, ymax=8.0,
  xmin=-0.5,xmax=8.0,
  axis line style={draw=none}
  ] \node[anchor=east] (source) at (axis cs:0,8.0){\small{$\T^2$}};
    \node[anchor=east] (source) at (axis cs:4.0,6.5){\small{$A_0$}};
    \draw[very thick](axis cs:0,0) -- (axis cs:8,0) -- (axis cs:8,8) -- (axis cs:0,8) -- (axis cs:0,0);
    \draw[thick](axis cs:2,1.236) -- (axis cs:7.5,0.618*7.5) -- (axis cs:5.8,2.236*7.5-1.618*5.8) -- (axis cs:0.3,4.472-1.618*0.3) -- (axis cs:2,1.236);
    \draw[thick] (axis cs:3.5,0.618*3.5) -- (axis cs:1.8,2.236*3.5-1.618*1.8);
    \draw[thick,pattern=horizontal lines](axis cs:2,1.236) -- (axis cs:3.5,0.618*3.5) -- (axis cs:1.8,2.236*3.5-1.618*1.8) -- (axis cs:0.3,4.472-1.618*0.3) -- (axis cs:2,1.236);
     cs:1.8,2.236*3.5-1.618*1.8) -- (axis cs:0.3,4.472-1.618*0.3) -- (axis cs:2,1.236);
     \draw[thick] (axis cs:0.7,-0.618*1.5+4.2664) -- (axis cs:6.2,0.618*4.0+4.2664);
     \draw[thick,pattern=vertical lines](axis cs:2,1.236) -- (axis cs:7.5,0.618*7.5) -- (axis cs:6.2,0.618*4.0+4.2664) -- (axis cs:0.7,-0.618*1.5+4.2664) -- (axis cs:2,1.236);
    \filldraw (axis cs:2.2,2.236*3.5-1.618*2.2) circle (2pt);
    \node[anchor=south] (source) at (axis cs:2.2,2.236*3.5-1.618*2.2){\small{$x$}};
  \end{axis}
\end{tikzpicture}
\caption{ $\xi_1(x)$ is the measure of the region shaded with horizontal  lines; $\xi_2(x)$ is the measure of the region shaded with vertical lines.}\label{fig:xi1}
\label{pic:xi_def}
\end{figure}
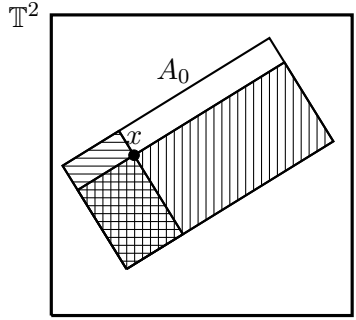


We introduce a new family of charts on $\T^2$.  For $v\in\Hom$, let $\tau_v$ denote the translation
$\tau_v(x)=x+v$. We then define a chart $\alpha_v$ with domain 
$\interior(A_0)-v$ by $\alpha_v=\xi\circ\tau_v$.
Since $\Hom$ is dense in $\T^2$, the collection
of charts covers all of $\T^2$. Our goal for the reminder of this subsection 
is to show that the family of charts $\{(\alpha_v,\interior(A_0)-v)\}_{v\in\Hom}$ 
forms a $C^{1+H}$-differentiable atlas on $\T^2$. We first prove a key lemma.

Let $v\in\Hom$ be such that $A_0\cap(A_0-v)$ has non-empty interior
and such that for any $x\in \interior(A_0\cap(A_0-v))$, the line segment 
joining $x$ and $x+v$ lies in $\interior(A_0)$ and is parallel to the unstable
direction. 
Using the notation from Section \ref{sec:coding} we consider the function 
$\xi_1(\pi(\omega))$ defined on 
$\pi^{-1}\big(\interior(A_0\cap (A_0-v))\big)\subset [0]$ in $ \Omega$ and study the limit
\begin{equation}\label{eq:def_ell}
    \ell(\omega):=\lim_{\omega'\to\omega}
\frac{\xi_1(\pi(\omega')+v)-\xi_1(\pi(\omega)+v)}{\xi_1(\pi(\omega'))-\xi_1(\pi(\omega))}.
\end{equation}
Here the limit is taken over those $\omega'$ such that $\xi_1(\pi(\omega'))\ne \xi_1(\pi(\omega))$, that
is those $\omega'$ such that $\pi(\omega')$ does not lie in the same local stable manifold as $\pi(\omega)$.
This is illustrated in Figure \ref{fig:quotient}.

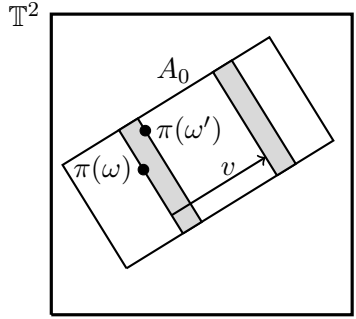
\begin{figure}[ht]
\begin{tikzpicture}
  \begin{axis} [clip=false,
  x=0.5cm,y=0.5cm,
  axis y line=left,
  axis x line=middle,
  ytick=\empty,
  xtick=\empty,
  ymin=0.0, ymax=8.0,
  xmin=-0.5,xmax=8.0,
  axis line style={draw=none}
  ] \node[anchor=east] (source) at (axis cs:0,8.0){\small{$\T^2$}};
    \node[anchor=east] (source) at (axis cs:4.0,6.5){\small{$A_0$}};
    \draw[very thick](axis cs:0,0) -- (axis cs:8,0) -- (axis cs:8,8) -- (axis cs:0,8) -- (axis cs:0,0);
    \draw[thick](axis cs:2,1.236) -- (axis cs:7.5,0.618*7.5) -- (axis cs:5.8,2.236*7.5-1.618*5.8) -- (axis cs:0.3,4.472-1.618*0.3) -- (axis cs:2,1.236);
    \draw[thick,fill=gray!30](axis cs:3.5,0.618*3.5) -- (axis cs:4,0.618*4) -- (axis cs:2.3,2.236*4-1.618*2.3) -- (axis cs:1.8,2.236*3.5-1.618*1.8) -- (axis cs:3.5,0.618*3.5);
      \filldraw (axis cs:2.45,2.236*3.5-1.618*2.45) circle (2pt);
    \node[anchor=east] (source) at (axis cs:2.45,2.236*3.5-1.618*2.45){\small{$\pi(\omega)$}};
     \filldraw (axis cs:2.5,2.236*4-1.618*2.5) circle (2pt);
   \node[anchor=west] (source) at (axis cs:2.5,2.236*4-1.618*2.5){\small{$\pi(\omega')$}};
    \draw[thick,fill=gray!30](axis cs:6,0.618*6) -- (axis cs:6.5,0.618*6.5) -- (axis cs:4.8,2.236*6.5-1.618*4.8) -- (axis cs:4.3,2.236*6-1.618*4.3) -- (axis cs:6,0.618*6);
    \draw[thick,->](axis cs:3.2,2.236*3.5-1.618*3.2) -- (axis cs:3.2+2.5,2.236*3.5-1.618*3.2+0.618*2.5);
   \node[anchor=east] (source) at (axis cs:3.2+2,2.236*3.5-1.618*3.2+0.618*2){\small{$v$}};
    
  \end{axis}
\end{tikzpicture}
\caption{The numerator and denominator in the limit are respectively 
the measures of the right and left shaded rectangles.}
\label{fig:quotient}
\end{figure}

\begin{lem}\label{lem:ell-properties}
Let $v\in\Hom$ be as described above. Then the limit $\ell(\omega)$, defined above,
exists for all $\omega$ in $\pi^{-1}\big(\interior(A_0\cap (A_0-v))\big)$ and
the function $\ell(\omega)$ is torus-H\"older on its domain. 
\end{lem}

\begin{proof}

Letting $R[\omega,\omega']$ be the rectangle bounded on the top and bottom by the boundary of $A_0$
and the left and right by the stable manifolds through $\pi(\omega)$ and $\pi(\omega')$, we see that
\begin{equation}\label{eq:lim_of_rect}
  \frac{\xi_1(\pi(\omega')+v)-\xi_1(\pi(\omega)+v)}{\xi_1(\pi(\omega'))-\xi_1(\pi(\omega))}
  =\frac{\mu(R[\omega,\omega']+v)}{\mu(R[\omega,\omega'])}
\end{equation}

We now apply the discussion of Section \ref{sec:R-N_deriv} to the case when $T$ 
is the toral automorphism $L$. For any $v\in \T^2$ homoclinic to 0 under $L$ 
the map $x\mapsto x+v$ is a (global) conjugating homeomorphism 
of $\T^2$.
It follows from Lemma  \ref{lem:rnderivHolder} that for an equilibrium state 
$\mu$ of a H\"{o}lder potential $\phi$ we have 
$$
\frac{d\mu(x+v)}{d\mu(x)}=\theta_v(x),
$$
where
\begin{equation}\label{eq:Radon-Nikodym}
    \theta_v(x)=\exp \left(\sum_{n\in\Z}\big[\phi(L^n(x+v))-\phi(L^n(x))\big]\right).
\end{equation}
Recall that by Lemma \ref{lem:rnderivHolder} the function $\theta_v\colon\T^2\to\R$ is H\"older continuous.

We can now rewrite (\ref{eq:def_ell}) as
\begin{align*}
\ell(\omega)&=\lim_{\omega'\to\omega}\frac{\mu(R[\omega,\omega']+v)}{\mu(R[\omega,\omega'])}\\
&=\lim_{\omega'\to\omega}\frac{\int_{R[\omega,\omega']}\theta_v(x)\,d\mu(x)}
{\int_{R[\omega,\omega']}1\,d\mu(x)}.
\end{align*}

We observe that $\pi^{-1}R[\omega,\omega']$ is a subset of $\Omega$ consisting of points $\zeta$
such that $\zeta_0^\infty$ are the non-negative coordinates of points lying between $\pi(\omega)$ and
$\pi(\omega')$. There is no restriction on the negative coordinates other than that $\zeta\in\Omega$ and
$\zeta_0=0$. Write $A^+[\omega,\omega']$ for $\{\zeta^+\in\Omega^+\colon \zeta^+\text{ are the non-negative
coordinates of a point in $R[\omega,\omega']$}\}$ and $A^-$ for 
$\{\zeta^-\in\Omega^-\colon\text{ $\zeta^-_{-1}0$
is legal in $\Omega$}\}$. We now apply Lemma \ref{lem:Ruelle}, giving
\begin{equation}
\ell(\omega)=\lim_{\omega'\to\omega}\frac
{\int_{A^-}\int_{A^+[\omega,\omega']}\varrho(\zeta)\theta_v(\pi(\zeta))\,d\nu^+(\zeta^+)\,d\nu^-(\zeta^-)}
{\int_{A^-}\int_{A^+[\omega,\omega']}\varrho(\zeta)\,d\nu^+(\zeta^+)\,d\nu^-(\zeta^-)}.
\label{eq:approxlom}
\end{equation}
Since $\varrho$ and $\theta_v\circ\pi$ are continuous, the integrands in the numerator and denominator
may be approximated for $\omega'$ close to $\omega$ by 
$\varrho(\zeta^-\omega^+)\theta_v(\pi(\zeta^-\omega^+))$
and $\varrho(\zeta^-\omega^+)$ respectively. Since these
new integrands 
don't depend on $\zeta^+$, the inner integrals of the
approximation to \eqref{eq:approxlom} are just the
product of the integrand and $\nu^+(A^+[\omega,\omega'])$. 
Since $\rho$ is strictly positive, cancelling the common factor, we now see that the limit exists, and
\begin{align*}
    \ell(\omega)&=\frac
{\int_{A^-}\varrho(\zeta^-\omega^+)\theta_v(\pi(\zeta^-\omega^+))\,d\nu^-(\zeta^-)}
{\int_{A^-}\varrho(\zeta^-\omega^+)\,d\nu^-(\zeta^-)}\\
&=\int_{A^-}\varrho(\zeta^-\omega^+)\theta_v(\pi(\zeta^-\omega^+))\,d\nu^-(\zeta^-),
\end{align*}
where the second equality follows from (\ref{eq:marginal}). 
Further, since $\varrho$ and $\theta_v\circ\pi$ are H\"older continuous functions on $\Omega$, we can see that
$\ell(\omega)$ is a H\"older continuous function of $\omega$ on $[0]$, depending only on the non-negative
coordinates of $\omega$.

In order to show that $\ell(\omega)$ is also torus-continuous, we consider $\omega$ belonging to the stable
manifold of 0 (so that $\pi(\omega)$, which we assumed to lie in
$\interior(A_0)$, lies on the boundary of two elements of $L^{-j}\mathcal P$
for some $j>0$: one on the left and one on the right). In this case, $p_+^{-1}(\pi(\omega))$ consists of
two elements, say $\omega^+$ and $\eta^+$. We will show that $\ell({\omega^-\omega^+})=
\ell({\omega^-\eta^+})$.

It will be convenient to find another
expression for $\ell(\omega)$ in which $\pi(\omega)$ is translated by another homoclinic vector
$\tilde{v}$ (which by Lemma \ref{lem:bracket} we can assume to be parallel to the
unstable direction and to satisfy $[\pi(\omega),\pi(\omega)+\tilde{v}]\subset\interior(A_0)$). 
Let $\tilde{R}[\omega,\omega']=R[\omega,\omega']+\tilde{v}$ and denote by 
$\tilde{A}^+[\omega,\omega']$ the set of future codes of points in the rectangle 
$\tilde R[\omega,\omega']$.

By similar arguments to those above and using the fact that $\theta_{v-\tilde{v}}(x)=\theta_{-\tilde{v}}(x)\theta_v(x-\tilde{v})$ 
which is immediate from the expression of the Radon-Nikodym derivative (\ref{eq:Radon-Nikodym}), we obtain
\begin{align*}
    \ell(\omega)&=
\lim_{\omega'\to\omega}\frac{\mu(\tilde{R}[\omega,\omega']+v-\tilde{v})}
{\mu(\tilde{R}[\omega,\omega']-\tilde{v})}\\
    &=\lim_{\omega'\to\omega}\frac
{\int_{A^-}\int_{\tilde{A}^+[\omega,\omega']}\varrho(\zeta)\theta_{-\tilde{v}}(\pi(\zeta))
\theta_v(\pi(\zeta)-\tilde{v})\,d\nu^+(\zeta^+)\,d\nu^-(\zeta^-)}
{\int_{A^-}\int_{\tilde{A}^+[\omega,\omega']}\varrho(\zeta)\theta_{-\tilde{v}}(\pi(\zeta))
\,d\nu^+(\zeta^+)\,d\nu^-(\zeta^-)},
\end{align*}
As before, taking a limit as $\omega'$ approaches
$\omega$, we see that
\begin{equation}\label{eq:comp1}
\ell(\omega^-\omega^+)=\frac
{\int_{A^-}\varrho(\zeta^-\tilde{\omega}^+)\theta_{-\tilde{v}}(\pi(\zeta^-\tilde{\omega}^+))
\theta_v(\pi(\zeta^-\omega^+))\,d\nu^-(\zeta^-)}
{\int_{A^-}\varrho(\zeta^-\tilde{\omega}^+)\theta_{-\tilde{v}}(\pi(\zeta^-\tilde{\omega}^+))
\,d\nu^-(\zeta^-)},
\end{equation}
where $\tilde{\omega}^+$ is the future coding of $\pi(\omega)+\tilde{v}$ corresponding to $\omega^+$.

Letting $\tilde{\eta}^+$ be the future coding of $\pi(\omega)+\tilde{v}$ corresponding to $\eta^+$ we get
\begin{equation}\label{eq:comp2}
\begin{split}
\ell({\omega^-\eta^+})&=\frac
{\int_{A^-}\varrho(\zeta^-\tilde{\eta}^+)\theta_{-\tilde{v}}(\pi(\zeta^-\tilde{\eta}^+))
\theta_v(\pi(\zeta^-\eta^+))\,d\nu^-(\zeta^-)}
{\int_{A^-}\varrho(\zeta^-\tilde{\eta}^+)\theta_{-v'}(\pi(\zeta^-\tilde{\eta}^+))\,d\nu^-(\zeta^-)}\\
&=\frac
{\int_{A^-}\varrho(\zeta^-\tilde{\eta}^+)\theta_{-\tilde{v}}(\pi(\zeta^-\tilde{\omega}^+))
\theta_v(\pi(\zeta^-\omega^+))\,d\nu^-(\zeta^-)}
{\int_{A^-}\varrho(\zeta^-\tilde{\eta}^+)\theta_{-\tilde{v}}(\pi(\zeta^-\tilde{\omega}^+))\,d\nu^-(\zeta^-)},
\end{split}
\end{equation}
where we used the facts $\pi(\zeta^-\tilde{\eta}^+)=\pi(\zeta^-\tilde{\omega}^+)$ 
and $\pi(\zeta^-\eta^+)=\pi(\zeta^-\omega^+)$.

Comparing \eqref{eq:comp1} and \eqref{eq:comp2}, we see that the only place where they differ is
that in the numerator and denominator, $\varrho(\zeta^-\tilde{\omega}^+)$ 
is replaced by $\varrho(\zeta^-\tilde{\eta}^+)$.
However if $\tilde{v}$ is chosen so that $\pi(\omega)+\tilde{v}$ 
does not lie on the stable boundary of any element
of $\bigvee_{0\le j<n}L^{-j}\mathcal P$, then $\tilde{\eta}^+$ and 
$\tilde{\omega}^+$ agree for at least $n$ symbols.
Since $\varrho$ is H\"older continuous,
$\varrho(\zeta^-\tilde{\eta}^+)/\varrho(\zeta^-\tilde{\omega}^+)$ is uniformly 
exponentially close to 1 as $\zeta^-$ runs over $A^-$.
It follows that $\ell(\omega)=\ell({\omega^-\zeta^+})$, so that $\ell$ is torus-continuous.
\end{proof}

We are now ready to establish that the atlas $\{(\alpha_v,\interior(A_0-v))\colon v\in\Hom\}$
is $C^{1+H}$.
We need to prove that for $v_0,v_1\in\Hom$ with the property that 
$\interior(A_0-v_0)\cap \interior(A_0-v_1)\ne\emptyset$, the 
map $\alpha_{v_1}\circ\alpha_{v_0}^{-1}$ is differentiable 
with H\"older continuous derivative.
In this case, observe $\alpha_{v_1}\circ\alpha_{v_0}^{-1}=
(\xi\circ\tau_{v_1})\circ(\xi\circ\tau_{v_0})^{-1}
=\xi\circ \tau_w\circ\xi^{-1}$, where $w=v_1-v_0\in\Hom$.

Using Lemma \ref{lem:bracket}, we write $w=v+u$, where $v$ is in the
unstable direction and $u$ is in the stable direction. Moreover, if both $x$ and $x+w$ are in $\interior(A_0)$, then the line segment joining $x$ and $x+v$
lies in $\interior(A_0)$, so that $v$ satisfies the conditions of Lemma \ref{lem:ell-properties}. Let $h_1$ be the H\"older continuous function on $\interior(A_0)\cap \interior(A_0-w)$ such that $\ell=h_1\circ\pi$ on their domain.

We now evaluate the derivative of $\xi\circ\tau_{w}\circ\xi^{-1}$ using the function $\ell$.
If $(a,b)$ and $(a,b')$ have the same first coordinate and are in the range
of $\xi\circ\tau_w\circ\xi^{-1}$, then we see from the definition of $\xi$ that $\tau_w\circ\xi^{-1}(a,b)$
and $\tau_w\circ\xi^{-1}(a,b')$ lie on the same stable manifold, so that the first coordinates of
$\xi\circ\tau_{w}\circ\xi^{-1}(a,b)$ and $\xi\circ\tau_w\circ\xi^{-1}(a,b')$ agree.
Similarly the second coordinates of $\xi\circ\tau_{w}\circ\xi^{-1}(a,b)$ and
$\xi\circ\tau_{w}\circ\xi^{-1}(a',b)$ agree, so that $\xi\circ\tau_{w}\circ\xi^{-1}(a,b)$
is of the form $(f_1(a),f_2(b))$. We see from the definition of $\ell$ that for $(a,b)$ in the domain of
$\xi\circ\tau_{w}\circ\xi^{-1}$, $f_1'(a)=h_1(\xi^{-1}(a,b))=h_1(\xi_1^{-1}(a)\cap
\xi_2^{-1}(b))$. Since $h_1$ is constant on local stable manifolds, this can also be written as
$h_1(\xi_1^{-1}(a))$. 

We verify that $f_1'$ is H\"older; an almost identical argument will show that $f_2'$ is H\"older.
Let $e_u$ be the unit unstable direction and $z$ be the
bottom left corner of $A_0$. Using $\iota(t)=\xi_1(z+te_u)$ we can write
$f_1'(a)=h_1(z+\kappa^{-1}(a)e_u)$. To show that $f_1'$ is H\"older, it therefore
suffices to show that $\kappa^{-1}$ is H\"older, which follows from an estimate of the form
$|\kappa(t')-\kappa(t)|\ge c|t-t'|^\beta$. We conclude the proof by establishing an estimate of this form.
Let $t'>t$ and let $n$ be such that $|t-t'|\ge 2\diam(\mathcal P)\lambda^{-n}$ 
(as before, $\lambda$ denotes the
expanding eigenvalue of the matrix defining $L$). Then between
the local stable manifolds through $z+t\,e_u$ and
$z+t'\,e_u$, there is at least one full element of $\bigvee_{j=0}^{n-1}L^{-j}\mathcal P$. 
By the Gibbs inequality,
these elements have measure at least $c'e^{-\delta n}$ for some $c'$ and $\delta$ that are independent
of $t$ and $t'$, so that
$|\kappa(t')-\kappa(t)|\ge c'e^{-\delta n}$. But from the bound on $|t-t'|$, 
we deduce $|\kappa(t')-\kappa(t)|\ge c|t-t'|^\beta$
for some $c$ and $\beta$ as required.

\subsection{Differentiability of $L$ with respect to the new atlas}\label{sec:diff_L}

We proved in Section \ref{sec:atlas} that the family of charts 
$\Xi=\{(\alpha_v, \interior(A_0)-v)\}_{v\in\Hom}$ form a $C^{1+H}$-differentiable 
atlas on $\T^2$. In this section we show that $L:(\T^2,\Xi)\to (\T^2,\Xi)$ is $C^{1+H}$. 

We first consider the case when $A_0\cap L^{-1}A_0$ has non-empty interior.
We claim that it suffices to establish that
$\xi\circ L\circ\xi^{-1}$ is $C^{1+H}$ on $\xi(A_0\cap L^{-1}A_0)$. 
To see this, let $v_0,v_1\in\Hom$ be such that the domain of $\alpha_{v_1}\circ L
\circ\alpha_{v_0}^{-1}$, i.e. $U:=(A_0-v_0)\cap L^{-1}(A_0-v_1)$, has non-empty interior.
Let $(a,b)\in \alpha_{v_0}(U)$ and write $(a,b)=\alpha_{v_0}(x)=\xi(v_0+x)$. Let 
$w\in\Hom$ be such that $x+v_0+w\in \interior(A_0\cap L^{-1}A_0)$. We now
see that on a neighbourhood of $(a,b)$
\begin{align*}
\alpha_{v_1}\circ L\circ \alpha_{v_0}^{-1}&=\xi\circ\tau_{v_1}\circ L\circ\tau_{-v_0}\circ\xi^{-1}\\
&=(\xi\circ\tau_{v_1-Lv_0-Lw}\circ \xi^{-1})\circ (\xi\circ L\circ\xi^{-1})
\circ (\xi\circ\tau_w\circ\xi^{-1}):
\end{align*}
$\xi\circ\tau_w\circ\xi^{-1}(a,b)=\xi(x+v_0+w)$; $\xi\circ L\circ \xi^{-1}(\xi(x+v_0+w))=\xi(Lx+Lv_0+Lw)\in \xi(A_0)$;
$\xi\circ \tau_{v_1-Lv_0-Lw}\circ\xi^{-1}(\xi(Lx+Lv_0+Lw))=\xi(Lx+v_1)=\alpha_{v_1}\circ L\circ \alpha_{v_0}^{-1}(a,b)$. 
Once we establish that $\xi\circ L\circ\xi^{-1}$ is $C^{1+H}$ on $\xi(A_0\cap L^{-1}A_0)$,
it will follow from the results of the
previous section that $\alpha_{v_1}\circ L\circ \alpha_{v_0}^{-1}$ is $C^{1+H}$ on a neighbourhood of $(a,b)$.



A similar argument to that in Section \ref{sec:atlas} 
shows that $\xi\circ L\circ \xi^{-1}(c,d)$ is of the form 
$(f_1(c), f_2(d))$ on its domain. We establish H\"older differentiability of 
$\xi\circ L\circ \xi^{-1}$ following the strategy 
of the previous section: first we show that $f_1'$ is shift-H\"older 
and then we verify that $f_1'$ is torus-continuous.
We compute 
\begin{equation}\label{eq:diff_L}
  f'_1(a)=\lim_{h\to 0}\frac{\xi\circ L\circ\xi^{-1}(a+h,b)-\xi\circ L\circ\xi^{-1}(a,b)}{h}.
\end{equation} From the definition of $\xi$ we see that $h$ is the $\mu$-measure 
of the rectangle in $A_0$ lying between the stable manifolds through $x$ and 
$x'=\xi^{-1}(a+h,b)$. Assuming that $h$ is small enough that $x'$ is also in 
$A_0\cap L^{-1}A_0$, we can write the numerator in the limit (\ref{eq:diff_L}) as the 
$\mu$-measure of the rectangle in $A_0$ lying between the stable manifolds 
through $L(x)$ and $L(x')$. We provide an illustration in Figure \ref{pic:diff_L} below.

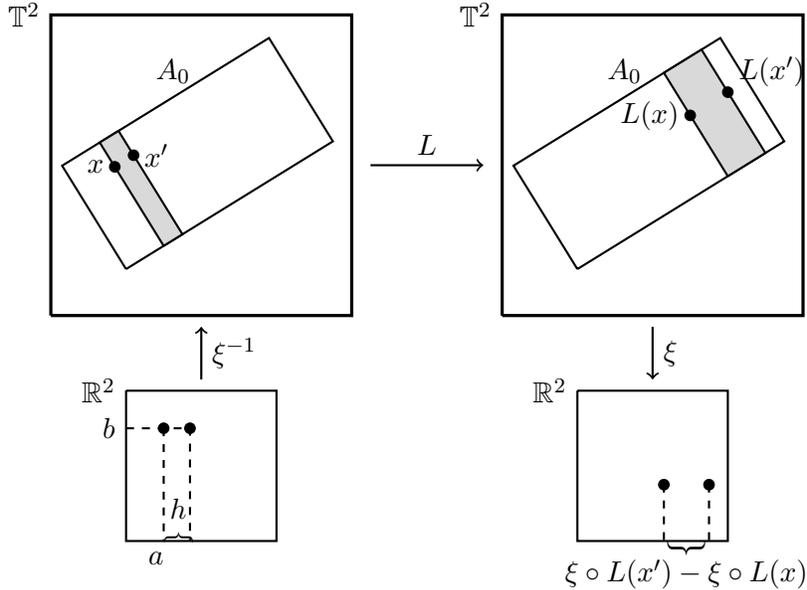
\begin{figure}[ht]
\begin{tikzpicture}
  \begin{axis} [clip=false,
  x=0.5cm,y=0.5cm,
  axis y line=left,
  axis x line=middle,
  ytick=\empty,
  xtick=\empty,
  ymin=-4.0, ymax=8.0,
  xmin=-0.5,xmax=18.0,
  axis line style={draw=none}
  ] \node[anchor=east] (source) at (axis cs:0,8.0){\small{$\T^2$}};
    \node[anchor=east] (source) at (axis cs:4.0,6.5){\small{$A_0$}};
    \draw[very thick](axis cs:0,0) -- (axis cs:8,0) -- (axis cs:8,8) -- (axis cs:0,8) -- (axis cs:0,0);
    \draw[thick](axis cs:2,1.236) -- (axis cs:7.5,0.618*7.5) -- (axis cs:5.8,2.236*7.5-1.618*5.8) -- (axis cs:0.3,4.472-1.618*0.3) -- (axis cs:2,1.236);
    \draw[thick,fill=gray!30](axis cs:3,0.618*3) -- (axis cs:3.5,0.618*3.5) -- (axis cs:1.8,2.236*3.5-1.618*1.8) -- (axis cs:1.3,2.236*3-1.618*1.3) -- (axis cs:3,0.618*3);
    \filldraw (axis cs:1.7,2.236*3-1.618*1.7) circle (2pt);
    \node[anchor=east] (source) at (axis cs:1.7,2.236*3-1.618*1.7){\small{$x$}};
    \filldraw (axis cs:2.2,2.236*3.5-1.618*2.2) circle (2pt);
    \node[anchor=west] (source) at (axis cs:2.2,2.236*3.5-1.618*2.2){\small{$x'$}};
    \draw[->,thick] (axis cs:8.5,4) -- (axis cs:11.5,4);
    \node[anchor=south] (source) at (axis cs:10.0,4.0){\small{$L$}};
    \node[anchor=east] (source) at (axis cs:12,8.0){\small{$\T^2$}};
    \node[anchor=east] (source) at (axis cs:16.0,6.5){\small{$A_0$}};
    \draw[very thick](axis cs:12,0) -- (axis cs:20,0) -- (axis cs:20,8) -- (axis cs:12,8) -- (axis cs:12,0);
    \draw[thick](axis cs:14,1.236) -- (axis cs:19.5,0.618*7.5) -- (axis cs:17.8,2.236*7.5-1.618*5.8) -- (axis cs:12.3,4.472-1.618*0.3) -- (axis cs:14,1.236);
    \draw[thick,fill=gray!30](axis cs:18,0.618*6) -- (axis cs:19,0.618*7) -- (axis cs:17.3,2.236*7-1.618*5.3) -- (axis cs:16.3,2.236*6-1.618*4.3) -- (axis cs:18,0.618*6);
    \filldraw (axis cs:17.0,2.236*6-1.618*5.0) circle (2pt);
    \node[anchor=east] (source) at (axis cs:17.0,2.236*6-1.618*5.0){\small{$L(x)$}};
    \filldraw (axis cs:18,2.236*7-1.618*6) circle (2pt);
    \node[anchor=west] (source) at (axis cs:18,6.5){\small{$L(x')$}};
    \node[anchor=east] (source) at (axis cs:2,-2.0){\small{$\R^2$}};
    \draw[thick](axis cs:2,-2) -- (axis cs:6,-2) -- (axis cs:6,-6) -- (axis cs:2,-6) -- (axis cs:2,-2);
    \draw[->,thick] (axis cs:4,-1.7) -- (axis cs:4,-0.3);
    \node[anchor=west] (source) at (axis cs:4,-1){\small{$\xi^{-1}$}};
    \filldraw (axis cs:3,-3) circle (2pt);
    \filldraw (axis cs:3.7,-3) circle (2pt);
    \node[anchor=east] (source) at (axis cs:2,-3.0){\small{$b$}};
    \node[anchor=north] (source) at (axis cs:2.8,-6){\small{$a$}};
    \node[anchor=south] (source) at (axis cs:3.39,-5.7){\small{$h$}};
    \node[rotate=90] (source) at (axis cs:3.4,-5.9){\small{$\}$}};
    \draw[thick, dashed](axis cs:2,-3.0) -- (axis cs:3.7,-3.0) -- (axis cs:3.7,-6);
    \draw[thick, dashed](axis cs:3,-3.0) -- (axis cs:3,-6);
    \node[anchor=east] (source) at (axis cs:14,-2.0){\small{$\R^2$}};
    \draw[thick](axis cs:14,-2) -- (axis cs:18,-2) -- (axis cs:18,-6) -- (axis cs:14,-6) -- (axis cs:14,-2);
    \draw[<-,thick] (axis cs:16,-1.7) -- (axis cs:16,-0.3);
    \node[anchor=west] (source) at (axis cs:16.0,-1){\small{$\xi$}};
    \filldraw (axis cs:16.3,-4.5) circle (2pt);
    \filldraw (axis cs:17.5,-4.5) circle (2pt);
    \draw[thick, dashed](axis cs:16.3,-4.5) -- (axis cs:16.3,-6);
    \draw[thick, dashed](axis cs:17.5,-4.5) -- (axis cs:17.5,-6);
    \node[rotate=-90] (source) at (axis cs:16.9,-6.2){\large{$\}$}};
     \node[anchor=north] (source) at (axis cs:16.9,-6.2){\small{$\xi\circ L(x')-\xi\circ L(x)$}};
  \end{axis}
\end{tikzpicture}
\caption{The $\mu$-measures of the shaded rectangles on the right and left are the numerator 
and the denominator in the limit (\ref{eq:diff_L}) respectively.}
\label{pic:diff_L}
\end{figure}


The derivative of $f_1$ can be represented symbolically on $[00]\subset\Omega$ as
$$
\ell(\omega)=\lim_{\omega'\to\omega}\frac{\mu(R[\sigma(\omega),\sigma(\omega')])}
{\mu(R[\omega,\omega'])},
$$
where, as before, $R[\omega,\omega']$ and $R[\sigma(\omega),\sigma(\omega')]$ 
are the rectangles bounded on the top and bottom
by the boundary of $A_0$ and on the sides by the stable manifolds through
$\pi(\omega)$, $\pi(\omega')$ and $L(\pi(\omega))$, $L(\pi(\omega'))$ respectively. 
Again, we observe that $\pi^{-1}(R[\omega,\omega'])=A^-\times A^+[\omega,\omega']$ 
and $\pi^{-1}(R[\sigma(\omega),\sigma(\omega')])=A^-\times A^+[\omega,\omega']$ 
where $A^-$, $A^+[\omega,\omega']$ are defined as in Section \ref{sec:atlas}. 
On the other hand, $\pi^{-1}R[\sigma(\omega),\sigma(\omega')]$ is a subset of 
$\Omega$ consisting of points $\zeta$ such that $\zeta_0^\infty$ are the non-negative 
coordinates of points in $L(R[\omega,\omega'])$ and there are no additional 
restrictions on the negative coordinates. Using Lemma \ref{lem:Ruelle} we obtain
$$
\ell(\omega)=\lim_{\omega'\to\omega}
\frac{\nu(A^-\times\sigma_+( A^+[\omega,\omega']))}
{\nu (A^-\times A^+[\omega,\omega'])}
=\lim_{\omega'\to\omega}\frac{{\nu}^+(\sigma_+(A^+[\omega,\omega']))}{{\nu}^+(A^+[\omega,\omega'])}.
$$
Since $\diam(A^+[\omega,\omega'])\to 0$ as $\omega'\to\omega$ and 
$\omega\in A^+[\omega,\omega']$, we conclude that $\ell(\omega)=\frac{1}{g(p_+(\omega))}$, 
where $g$ is the $g$-function for measure $\nu^+$. Since $g$ is strictly positive 
and H\"older on $\Omega^+$, $\ell$ is H\"older on $\Omega$.

To prove that $\ell$ is torus continuous suppose that $x=\pi(\omega)$ lies on the stable 
manifold boundary of two elements of the partition $L^{-j}\cP$ for some $j\in\N$. 
Let $\omega^-\omega^+$ and $\omega^-\eta^+$ be two different symbolic representations 
of $x$. To show that $\ell(\omega^-\omega^+)=\ell(\omega^-\eta^+)$ we apply the 
same steps as in Section \ref{sec:atlas}. For any $N\in \N$, let $v\in\Hom$ be parallel 
to the unstable direction satisfying $x+v\in \interior(A_0\cap L^{-1}A_0)$ and 
$x+v\notin \bigcup_{|k|<N} \partial L^k\cP$. 
Let $\tilde{R}[\omega,\omega']=R[\omega,\omega']+v$ and denote by 
$\tilde{A}^+[\omega,\omega']$ the set of future coordinates of points in 
$\tilde{R}[\omega,\omega']$. Using the expression for the Radon-Nikodym 
derivative (\ref{eq:Radon-Nikodym}) we obtain
$$
\mu(R[\omega,\omega'])=\int_{A^-}\int_{\tilde{A}^+[\omega,\omega']}
\varrho(\zeta^-\zeta^+)\theta_{-v}(\pi(\zeta^-\zeta^+))\,d\nu^+(\zeta^+)\,d\nu^-(\zeta^-).
$$
Similarly, let $\tilde{R}[\sigma(\omega),\sigma(\omega')]=R[\sigma(\omega),\sigma(\omega')]+L(v)$ 
and obtain
$$
\mu(R[\sigma(\omega),\sigma(\omega')])=\int_{A^-}\int_{\sigma^+(\tilde{A}^+[\omega,\omega'])}
\varrho(\zeta^-\zeta^+)\theta_{-L(v)}(\pi(\zeta^-\zeta^+))\,d\nu^+\,d\nu^-.
$$
Consider $\omega=\omega^-\omega^+$ and denote be $\tilde{\omega}^+$ the corresponding 
future coding of $\pi(\omega)+v$. By continuity of $\varrho$ and $\theta$ for 
each $\zeta^-$ the inner integral of $\mu(R[\omega,\omega'])$ is approximately 
$\varrho(\zeta^-\tilde{\omega}^+)
\theta_{-v}(\pi(\zeta^-\sigma^+(\tilde{\omega}^+)))
\nu^+(\tilde{A}^+[\omega,\omega'])$ 
and similarly the inner integral of $\mu(R[\sigma(\omega),\sigma(\omega')])$ 
is approximately $\varrho(\zeta^-\tilde{\omega}^+)
\theta_{-L(v)}(\pi(\zeta^-\sigma^+(\tilde{\omega}^+)))
\nu^+(\tilde{A}^+[\sigma(\omega),\sigma(\omega')])$ 
whenever $\zeta^+$ is close enough to $\tilde{\omega}^+$. As $\omega'\to\omega$ the diameter of 
$\tilde{A}^+[\omega,\omega']$ tends to zero while $\tilde{\omega}^+\in 
\tilde{A}^+[\omega,\omega']$, so that 
$\frac{{\nu}^+(\sigma_+(\tilde{A}^+[\omega,\omega']))}{{\nu}^+(\tilde{A}^+[\omega,\omega'])}
\to \frac{1}{g(\tilde{\omega}^+)}$. Therefore,
\begin{equation*}
\ell(\omega^-\omega^+)=\frac{\int_{A^-}\varrho(\zeta^-\sigma^+(\tilde{\omega}^+))
\theta_{-L(v)}(\pi(\zeta^-\sigma^+(\tilde{\omega}^+)))\,d\nu^-(\zeta^-)}
{\int_{A^-}\varrho(\zeta^-\tilde{\omega}^+)\theta_{-v}(\pi(\zeta^-\tilde{\omega}^+))
\,d\nu^-(\zeta^-)}\cdot\frac{1}{g(\tilde{\omega}^+)}
\end{equation*}
Letting $\tilde{\eta}^+$ be the future coding of $\pi(\omega)+v$ corresponding to $\eta^+$ we get
\begin{equation*}
\ell(\omega^-\eta^+)=\frac{\int_{A^-}\varrho(\zeta^-\sigma^+(\tilde{\eta}^+))
\theta_{-L(v)}(\pi(\zeta^-\sigma^+(\tilde{\eta}^+)))\,d\nu^-(\zeta^-)}
{\int_{A^-}\varrho(\zeta^-\tilde{\eta}^+)\theta_{-v}(\pi(\zeta^-\tilde{\eta}^+))\,d\nu^-(\zeta^-)}
\cdot\frac{1}{g(\tilde{\eta}^+)}.
\end{equation*}
Note that since $\omega^-\tilde{\omega}^+$  and $\omega^-\tilde{\eta}^+$ are two symbolic codings of 
a single point $x+v$ in $\interior(\pi([00]))$, $\sigma(\omega^-\tilde{\omega}^+)$ and 
$\sigma(\omega^-\tilde{\eta}^+)$ are two symbolic codings of the point $L(x+v)$ in 
$\interior(\pi([0])$. Hence, $\pi(p_+^{-1}(\sigma_+(\tilde{\omega}^+)))$ and 
$\pi(p_+^{-1}(\sigma_+(\tilde{\eta}^+)))$ are the same local stable manifold inside $\pi([0])$. 
Both points $\pi(\omega^-\sigma_+\tilde{\omega}^+)$ and $\pi(\omega^-\sigma_+\tilde{\eta}^+)$ 
lie on the intersection of this local stable manifold and the local 
unstable manifold $\pi(p_-^{-1}(\omega^-))$ inside $\pi([0])$, so they must coincide.

Repeating the argument at the end of Section \ref{sec:atlas} completes the proof. 
Since $x+v$ is not on the boundary of the partition $\bigvee_{0\le k<N}L^{-k}\cP$, 
$\tilde{\omega}^+$ and $\tilde{\eta}^+$ agree on at least $N$ symbols. Now H\"older 
continuity of $\varrho$ and $g$ implies that the ratio 
$\ell(\omega^-\omega^+)/\ell(\omega^-\eta^+)$ can be made arbitrarily close to one when 
by choosing $N$ sufficiently large, so that $\ell$ is torus continuous.

So far, we have completed the proof that $L$ is $C^{1+H}$ in the new charts
in the case that $A_0\cap L^{-1}A_0$ has non-empty interior. 
An essentially identical argument shows that if $A_0\cap L^{-n}A_0$ has non-empty interior,
then $L^n$ is $C^{1+H}$ in the new charts. (The only modification is that the $g$-function
has to be replaced by $g^{(n)}$ defined by 
$g^{(n)}(x)=g(x)g(\sigma(x))\cdots g(\sigma^{n-1}x)$). 
Since Anosov automorphisms are topologically mixing, $A_0\cap L^{-n}A_0$ has non-empty
interior for all sufficiently large $n$. In particular there is $n$ such
that $L^n$ and $L^{n+1}$ are both $C^{1+H}$ diffeomorphisms. It follows that
$L=(L^n)^{-1}\circ L^{n+1}$ is $C^{1+H}$ as required.




\subsection{Cohomology of $\phi$ and the geometric potential of $L$ in the new atlas.}

\begin{lem}\label{lem:rectarea}
Let $L$ and $\mathcal P$ be as above. There exist $\gamma>0$ and $k>0$ such that
if $R\subset A_0$ is of the form 
$R=\pi(C_-\times S)$ where $C_-$ is an $n$-cylinder in $\Omega_-$ and $S\subset[0]\subset\Omega_+$,
then $\mu(R)\le ke^{-\gamma n}\mu(\pi\circ p_+^{-1}S)$.
\end{lem}

The proof is an application of the product structure outlined in Section \ref{sec:coding} 
together with the fact that $\nu^-$ is a $g$-measure
with $g_-$ bounded away from 1. 

\begin{lem}\label{lem:expanding}
The map $L$ is expanding in the unstable direction in the new coordinate system:
for any finite sub-atlas there exists $n\in\N$ such that 
for any $x\in \T^2$ and for any charts in the sub-atlas containing $x$ and $L^nx$ 
respectively, $D_uL^n>2$ when computed in the respective charts. 
\end{lem}
\begin{proof}
Let the finite sub-atlas be $\{\alpha_{u_1},\ldots,\alpha_{u_N}\}$.  
Let $M$ and $M'$ be positive constants such that $\tfrac 1M\le \theta_{u_i}(x)\le M$ 
for $1\le i\le N$ and  all of the maps $\alpha_{u_i}\circ\alpha_{u_j}^{-1}$ 
have derivatives between $M'$ and ${M'}^{-1}$
when $1\le i,j\le N$. 
By compactness, there exists $\delta>0$ such that for each $x\in\T^2$, there exists an $i$
with $x+u_i\in A_0$ and $d(x+u_i,\partial A_0)>\delta$. Let $n$ be a fixed integer sufficiently
large that $\lambda^{-n}<\delta$ and also satisfying  $e^{\gamma n}>2kM^2{M'}^2$, where  $\lambda$ is the
expanding eigenvalue of $L$ and $k$, $\gamma$ are as in Lemma \ref{lem:rectarea}.

Let $u,v\in\{u_1,...,u_N\}$ be such that $x+u$ and $L^nx+v$ both lie in $\text{int}_\delta(A_0)$.
Let $B_1+u$ be a rectangle in $A_0$ whose projection in $A_0$ onto the stable direction is
all of the stable manifold segment defining $A_0$ and whose unstable projection in $A_0$ 
is sufficiently narrow that 
$L^nB_1+v\subset A_0$. Let $B_2$ be the rectangle in $A_0$ whose projection onto the stable direction
is the stable manifold segment defining $A_0$ and whose unstable projection is the same as that of
$L^nB_1+v$. Then by Lemma \ref{lem:rectarea},
\begin{equation}\label{eq:prodbound}
  \mu(L^nB_1+v)\le ke^{-\gamma n}\mu(B_2).
\end{equation}

We then have
\begin{equation}\label{eq:RNbounds}
  \begin{split}
     \mu(L^nB_1+v)&\ge \tfrac 1M\mu(L^n B_1)\\
     \mu(L^n B_1)&=\mu(B_1)\\
     \mu(B_1)&\ge \tfrac 1M\mu(B_1+u).
  \end{split}
\end{equation}
Combining equations \eqref{eq:prodbound} and \eqref{eq:RNbounds}, by the choice of $n$
we see 
$$
\mu(B_2)\ge \frac{e^{\gamma n}}{kM^2}\mu(B_1+u)\ge 2M'^2\mu(B_1+u).
$$
    
Shrinking $B_1$ so that $B_1+u$ shrinks to the segment of the stable manifold of $x$ lying in $A_0$, we
deduce the unstable derivative of $L^n$ in the $(\alpha_u,\alpha_v)$ charts is at least $2{M'}^2$. 
Now if $u'$ and $v'$ are such that $\alpha_{u'}$ and $\alpha_{v'}$ are arbitrary charts in the sub-atlas
containing $x$ and $L^nx$ in their domain then the unstable derivative of $L^n$ in the
$(\alpha_{u'},\alpha_{v'})$ charts is at least 2. This completes the proof.    
\end{proof}

\begin{lem}\label{lem:Livsiccond}
    There exists $M>0$ such that
    for any $n\in\N$, any cylinder set $C$ in $\Omega$ of the form $[0a_1\ldots a_{n-1}0]$, 
    and any $\omega\in C$
    we have 
    $$
    \frac 1M \le |D_uL^n(\pi(\omega))|\cdot\exp(S_n\phi(\pi(\omega)))\le M,
    $$
    where the unstable derivative of $L$ is computed using the new charts.
\end{lem}

\begin{proof}
The proof is based on a standard argument that the fibre maps of uniformly expanding 
maps have bounded distortion (see e.g. \cite[Chapter III]{Mane}). 
Suppose that $x,y$ are points in $A_0$ which lie on the same local unstable manifold 
and are such that $x=\pi(\omega), y=\pi(\eta)$ with 
$\omega,\eta\in [0a_1\ldots a_{n-1}0]\subset \Omega$. Recall from Section 
\ref{sec:diff_L} that in terms of charts of the new atlas, the map $L$ has the form 
$\alpha_{v_1}\circ L\circ \alpha_{v_0}^{-1}(a,b)=(f_1(a), f_2(b))$ where the functions 
$f_1$ and $f_2$, which depend on the choice of $v_0$ and $v_1$, are differentiable 
with H\"older continuous derivatives. Since $\sigma^n(\omega),\sigma^n(\eta)\in [0]$ we 
have that both $L^n(x)$ and $L^n(y)$ are in $A_0$ and hence 
$d(L^j(x),L^j(y))\le \lambda^{-(n-j)}$ for $0\le j<n$, where $\lambda$ is the 
expanding eigenvalue of $L$ (and here the distance is computed using the original
metric). Denote by $f_{1,j}$ the first component of $L$ 
computed in the charts corresponding to $L^j x$ and $L^{j+1}x$. 
Applying the chain rule we see that 
$$
\left|\frac{D_uL^n(x)}{D_uL^n(y)}\right|=\prod_{j=0}^{n-1}
\left|\frac{f'_{1,j}(L^jx)}{f'_{1,j}(L^jy)}\right|.
$$
It follows from H\"older continuity of the derivatives and Lemma \ref{lem:expanding} that there are 
$K>0$ and $\gamma\in(0,1)$ such that for $0\le j<n$
$$
\left|\frac{f_{1,j}(L^j(x))}{f_{1,j}(L^j(y))}\right|
\le 1+Kd(L^j(x),L^j(y))^\gamma\le 1+K\lambda^{-(n-j)}.
$$
Setting $M=\prod_{j=1}^\infty(1+K\lambda^{-j})$ we obtain that $|D_uL^n(x)|/|D_uL^n(y)|\le M$ 
for all $x,y$ lying in a segment of the local unstable manifold contained in a single partition element.

Now suppose $C$ is a cylinder set 
$[a_0a_1\ldots a_n]$ in $\Omega$ with $a_0=a_n=0$.
Let $\omega^-$ be a compatible past and set $U=\pi(\omega^-C)$, a piece of unstable manifold
that is mapped bijectively by $L^n$ onto a fibre of the unstable manifold crossing the 
partition element $A_0$. 
By the mean value theorem, the length (in the new charts) of $L^n U$ (which is the same as the width of the 
0 partition element) is the product of the length of $U$ and the unstable derivative 
at some point $u\in U$. Since coordinates (and hence lengths) in the unstable direction are computed
using by $\mu$-measures the measure $\mu=\pi_\ast \nu$, this gives, for any $\omega\in U$,
$$
\frac 1M\le |(D_uL^n)(\pi(\omega))|\cdot \nu(C)\le M.
$$
Now applying the Bowen definition \cite{Bowen} for the Gibbs state $\nu$ of the potential 
$\phi\circ\pi$  together with the fact that $P_{\rm top}(\sigma,\phi\circ\pi)=0$,
$$
\frac {1}{M'} \exp(S_n\phi(\pi(\omega)))\le \mu(C)\le M'\exp(S_n\phi(\pi(\omega)))
$$
Substituting in the previous inequality gives the required statement. 
\end{proof}

\begin{lem}
Let $\phi$ be as in the statement of Theorem \ref{thm:main}, and let 
the charts be constructed as above. Then potential $\phi(x)$ is cohomologous to 
$-\log |D_uL(x)|$, where the unstable derivative is computed using the 
new charts. 
\end{lem}

\begin{proof}
We rely on Liv\v{s}ic's theorem \cite{Livsic}: if $T$ is a hyperbolic dynamical system and
$\psi$ is a H\"older continuous function such that $S_n\psi(p)=0$ whenever
$T^np=p$, then $\psi$ is a coboundary (with H\"older continuous transfer function). 

As a corollary, if $\Omega$ is a mixing subshift of finite type and 
there exists an $M$ such that $|S_{n+1}\psi(\omega)|\le M$ whenever $\omega\in [0]$
and $\sigma^n\omega\in[0]$, then $\psi$ is a H\"older coboundary. 

Lemma \ref{lem:Livsiccond} shows that the function $\psi(\omega)=\chi\circ\pi(\omega)$ where
$\chi(x)=\log |D_uL(x)|+\phi(x)$ satisfies the hypothesis
of this corollary of Liv\v{s}ic's theorem, so that $\psi$ is a H\"older coboundary. 
It follows that $\chi$ sums to zero around any periodic orbit in $\T^2$,
so that $\chi$ is also a H\"older coboundary, using Liv\v{s}ic's theorem again. 
\end{proof}

\section{Application to the smooth conjugacy problem.}
In this section we explicitly construct a countable family of H\"older potentials 
in the homotopy class of the toral automorphism $L$ whose geometric potentials 
have identical pressure functions, yet they are not $C^1$ conjugate.
\begin{lem}\label{lem:pressurelift}
    Let $L$ be an automorphism of $\T^2$, let $k\in\N$ and let $M_k(x)=kx\bmod 1$.
    Then for any continuous function $\phi$ on $\T^2$ 
    $$P_{\rm top}(L,\phi)=P_{\rm top}(L,\phi\circ M_k).$$
    \end{lem}

\begin{proof}
We use the topological definition of pressure:
$$
P_{\rm top}(L,\phi)=\lim_{\epsilon\to 0}\limsup_{n\to\infty}
\frac 1n\log\sup\left\{\sum_{x\in E}e^{S_n\phi(x)}: E \text{ is $(n,\epsilon)$-separated}\right\},
$$
where a subset $E$ of $\T^2$
is \emph{$(n,\epsilon)$-separated} (with respect to $L$) if for any distinct elements $x,y\in E$, there exists
$0\le j<n$ such that $d(L^jx,L^jy)\ge \epsilon$. 
 
Denote $\phi_k=\phi\circ M_k$. We first show that $P_{\rm top}(L,\phi_k)\ge P_{\rm top}(L,\phi)$. 
Let $E$ be an $(n,\epsilon)$-separated subset of $\T^2$. 
We define a subset $E'$ of $\T^2$ by $$E'=M_k^{-1}(E)=\{(x+\mathbf n)/k\colon 
x\in E,\mathbf n\in \{0,\ldots,k-1\}^2\}$$
and claim $E'$ is $(n,\frac\epsilon k)$-separated. 
In the case when $x\in E$ and $\mathbf m,\mathbf n$ are distinct elements of $\{0,\ldots,k-1\}^2$,
we claim $d(L^i(\frac{x+\mathbf m}k),L^i(\frac{x+\mathbf n}k))\ge \frac 1k$ for each $i$.
Since $L$ is an automorphism,
it suffices to show that $d(L^i(\frac{\mathbf p}k),0)
\ge \frac1k$ for each 
$\mathbf p\in\{0,\frac 1k,\ldots,\frac{k-1}k\}^2\setminus\{(0,0)\}$ and $i\in\N$. 
Since the matrix $A$ defining $L$ has an inverse with integer entries, it is not hard to see that $L$
is a permutation of the points $\{0,\ldots,\frac{k-1}k\}^2$. Since $L$ is injective, it
follows that $d(L^i(\frac{\mathbf p}k),0)\ge \frac1k$ for each $i$. 
In the case when $x,y$ are distinct elements of $E$
and $\mathbf m,\mathbf n$ are elements of $\{0,\ldots,k-1\}^2$ (not necessarily distinct),
letting $u=\frac{x+\mathbf m}k$ and
$v=\frac{y+\mathbf n}k$, we have 
$$
d(L^iu,L^iv)\ge \tfrac 1kd(M_k(L^iu),M_k(L^iv))=\tfrac 1kd(L^ix,L^iy).
$$ 
Since $\max_{i<n}d(L^ix,L^iy)\ge\epsilon$, it follows that $\max_{i<n}d(L^iu,L^iv)\ge\frac\epsilon k$.
Hence we have established that $E'$ is $(n,\frac\epsilon k)$-separated as required. 

Note $S_n\phi_k(\frac{x+\mathbf m}k)=S_n\phi(x)$ for each $x\in E$ and 
$\mathbf m\in \{0,\ldots,k-1\}^2$. Therefore
\begin{align*}
\sup&\left\{\sum_{x\in E}e^{S_n\phi_k(x)}: E \text{ is $\textstyle(n,\frac{\epsilon}{k})$-separated}\right\}\\
&\qquad\ge k^2\sup\left\{\sum_{x\in E}e^{S_n\phi(x)}: E \text{ is }(n,\epsilon)\text{-separated}\right\},
\end{align*} 
which gives $P_{\rm top}(L,\phi_k)\ge P_{\rm top}(L,\phi)$. 

For the converse inequality, we first claim that any $u,v\in \T^2$ and for any 
positive $\epsilon<1/(2k\|A\|)$, the following implication holds:
\begin{equation}\label{eq:latticeind}
d(u,v)<\epsilon\text{ and }d(M_k(Lu),M_k(Lv))<k\epsilon\text{ implies } d(Lu,Lv)<\epsilon.
\end{equation}
Again, by the linearity of $L$, it suffices to show that if $d(u,0)<\epsilon$ and $d(M_k(Lu),0)<k\epsilon$
then $d(Lu,0)<\epsilon$. To verify this claim, suppose $d(u,0)<\epsilon$. By the choice of $\epsilon$,
$d(Lu,0)<\frac 1{2k}$ so that $0$ is the closest element of $M_k^{-1}\{0\}$ to $Lu$. 
Since $d(M_k(Lu),0)<k\epsilon$,
the fact that $M_k$ locally expands distances by a factor of $k$ implies that 
$d(Lu,0)<\epsilon$ as required. 
    
Let $\epsilon<\frac{1}{2k\|A\|}$ and let $E'$ be an $(n,\epsilon)$ separated set in $\T^2$. 
We define a relation $R$ on $E'$ by
$$
u R v\quad \Leftrightarrow\quad\max_{0\le i<n}d(L^iM_k(u),L^iM_k(v))<\frac\epsilon{2k}.
$$
Equivalently $u R v$ iff $\max_{0\le i<n}d(M_k(L^iu),
M_k(L^iv))<\frac\epsilon{2k}$, since $L\circ M_k=M_k\circ L$.
We then take the transitive closure of $R$ to form an equivalence relation $\sim$ on $E'$. That is,
$u\sim v$ if there exist $u_0,u_1,\ldots,u_l$ with $u_0=u$, $u_l=v$ and $u_{i-1} R u_{i}$ 
for $i=1,\ldots,l$. 
We claim that each $\sim$-equivalence class has at most $k^2$ elements. We prove this by contradiction. 
Suppose $C$ is a $\sim$-equivalence class containing at least $k^2+1$ elements. 
We construct a subset $D$ of cardinality exactly $k^2+1$ such that there is a path between any 
two elements of $D$
using steps in $R$. To see this, fix an initial element of $u_0$ of $C$, 
enumerate the other elements of $C$ and for each such element $u$,
find an $R$-path, from the definition of $\sim$ connecting $u_0$ to $u$. 
We now build $D$ by adding the elements of the paths 
one at a time until the cardinality is exactly $k^2+1$. 
(At each step when a vertex is to be included, $D$ may either increase by one
element if the vertex is new; or remain the same if the vertex has already been added.) 
By the construction, each element of $D$ is connected by $R$ to a previous element of $D$.

Let $D=\{u_0,\ldots,u_{k^2}\}$. By the triangle inequality and the definition of $R$,
$d(M_k(L^i(u_0)),M_k(L^i(u_j)))<\frac{k\epsilon} 2$ for each $j$
(since we can get from $u_0$ to $u_j$ along an $R$ path of length at most $k^2$). 
In particular, $d(M_k(u_0),M_k(u_j))<\frac{k\epsilon} 2$ for each $j$. 
Using the fact that $M_k$ locally expands distances by a factor of $k$, for each $0\le j\le k^2$, $u_j$ 
differs from $u_0$ by an element of $M_k^{-1}\{0\}=\{0,\frac1k,\ldots,\frac{k-1}k\}^2$ 
plus a term of size at most $\frac{\epsilon} 2$. 
By the pigeonhole principle, there exist $0\le j<j'\le k^2$ such that $u_j$ 
and $u_{j'}$ differ by at most $\epsilon$. 
Since $u_j\sim u_{j'}$, we see that $d(L^iM_k(u_j),L^iM_k(u_{j'}))<k\epsilon$ for $i=0,\ldots,n$. 
Applying \eqref{eq:latticeind} inductively we see $d(L^i u_j,L^iu_{j'})<\epsilon$ for $i=0,\ldots,n$. 
This contradicts the initial assumption that $E'$ was $(n,\epsilon)$-separated. 

Hence we have shown that each $\sim$-equivalence class in $E'$ has at most $k^2$ elements. 
Let the equivalence classes be $C_1,\ldots,C_M$; and for each equivalence class, pick
$u_i\in C_i$ for which $S_n\phi_k(u_i)$ is maximal in the equivalence class. 
We now have
$$
\sum_{u\in C_i}\exp (S_n\phi_k(u))\le k^2\exp (S_n\phi_k(u_i)).
$$
Summing over the equivalence classes, we obtain
$$
\sum_{u\in E'}\exp(S_n\phi_k(u))\le k^2\sum_{i=1}^M\exp(S_n\phi_k(u_i)).
$$
Let $x_i=M_k(u_i)$ for each $i$. Since $S_n\phi_k(u_i)=S_n\phi(x_i)$, 
rearranging the above inequality gives
$$
\sum_{i=1}^M\exp(S_n\phi(x_i))\ge \frac1{k^2}\sum_{u\in E'}\exp(S_n\phi_k(u)).
$$
Finally, we claim that $\{x_1,\ldots,x_M\}$ is $(n,\frac\epsilon{2k})$-separated. 
If not then, there there exist $j,l$ such that $d(L^ix_j,L^ix_l)<\frac\epsilon{2k}$ for $i=0,\ldots,n-1$. 
Then, since $x_j=M_k(u_j)$ and $x_l=M_k(u_l)$, we see from the definition of $R$ that  $u_j R u_l$. This
contradicts the assumption that the $u_i$'s belong to distinct equivalence classes. 

Hence we have shown

\begin{align*}
\sup&\left\{\sum_{x\in E}e^{S_n\phi(x)}: E \text{ is $\textstyle(n,\frac{\epsilon}{2k})$-separated}\right\}\\
&\qquad\ge \frac{1}{k^2}\sup\left\{\sum_{u\in E'}e^{S_n\phi_k(u)}: 
E \text{ is }(n,\epsilon)\text{-separated}\right\},
\end{align*} 
    It follows that $P_{\rm top}(L,\phi)\ge P_{\rm top}(L,\phi_k)$ as required.    
\end{proof}

\begin{proof}[Proof of Theorem \ref{thm:Federico}]
Let $L$ be the Anosov automorphism of the torus given by the matrix $\begin{pmatrix}1&1\\1&0\end{pmatrix}$. 
Note that $(\frac 12,0),(\frac 12,\frac 12),(0,\frac 12)$ is the unique period 
3 orbit of $L$. Let $\phi$ be a H\"older
continuous function of the torus 
of pressure 0 such that $\phi(0,0)\ne \frac 13(\phi(\frac 12,0)+\phi(\frac 12,\frac 12)+\phi(0,\frac 12))$
and let $\phi_2(x)=\phi(2x)$ as above. Then 
$$\textstyle\frac 13(\phi_2(\frac 12,0)+\phi_2(\frac 12,\frac 12)+\phi_2(0,\frac 12))
=\phi(0,0)\ne \frac 13(\phi(\frac 12,0)+\phi(\frac 12,\frac 12)+\phi(0,\frac 12)).$$ 
We conclude the proof by showing that if $T$ and $T_2$ are the area-preserving 
Anosov diffeomorphisms obtained from $\phi$ and $\phi_2$ respectively
as in Theorem \ref{thm:main},
then $T$ and $T_2$ are not conjugate, but they satisfy $P_{\rm top}(T,-sD_uT)=
P_{\rm top}(T_2,-sD_u T_2)$ for all $s\in\R$. 

Let $h$ be the conjugacy between $T$ and $L$ obtained in the proof of Theorem \ref{thm:main}.  
Similarly, let $h_2$ be the conjugacy between $T_2$ and $L$. 
The theorem guarantees that $-\log |D_u T|\circ h$ is cohomologous to $\phi$
and $-\log|D_u T_2|\circ h_2$ is cohomologous to $\phi_2$. Let $p=h(\frac 12,0)$ and notice that
$\{p,Tp,T^2p\}$ is the unique period 3 orbit of $T$. Similarly let $p_2=h_2(\frac 12,0)$ so that
$\{p_2,T_2p_2,T_2^2p_2\}$ is the unique period 3 orbit of $T_2$. Since $-\log |D_uT|\circ h$
is cohomologous to $\phi$, we see that $$|D_uT^3(p)|=|D_uT^3(Tp)|=|D_uT^3(T^2p)|=
e^{\phi(\frac 12,0)+\phi(\frac 12,\frac 12)+\phi(0,\frac 12)},$$
while $$|D_uT_2^3(p_2)|=e^{\phi_2(\frac 12,0)+\phi_2(\frac 12,\frac 12)+\phi_2(0,\frac 12)}=e^{3\phi(0,0)}.$$ 
Since differentiable conjugacies preserve unstable multipliers, we see that $T$ and $T_2$ are not
differentiably conjugate. 

However, 
\begin{align*}
    P_{\rm top}(T,-s\log|D_u T|)&=P_{\rm top}(hLh^{-1},-s\log|D_u T|)\\
    &=P_{\rm top}(L,-s\log|D_u T|\circ h)\\
    &=P_{\rm top}(L,-s\phi)
\end{align*}
and similarly $P_{\rm top}(T_2,-s\log|D_u T_2|)=P_{\rm top}(L,-s\phi_2)$. By Lemma \ref{lem:pressurelift},
$P_{\rm top}(L,-s\phi)=P_{\rm top}(L,-s\phi\circ M_2)=P_{\rm top}(L,-s\phi_2)$ for all $s\in\R$ so that 
$P_{\rm top}(T,-s\log|D_u T|)=P_{\rm top}(T_2,-s\log|D_u T_2|)$ for all $s$.
\end{proof}

\end{document}